\documentclass[11pt, reqno]{amsart}
\usepackage{amssymb}
\usepackage{verbatim}

\newtheorem{thm}{Theorem}[section]
\newtheorem{lemma}[thm]{Lemma}

\newtheorem{prop}[thm]{Proposition}

\newtheorem{defi}[thm]{Definition}

\theoremstyle{remark}

\theoremstyle{definition}

\numberwithin{equation}{section}

\DeclareMathOperator{\Disc}{Disc}
\DeclareMathOperator{\Conf}{Conf}

\DeclareMathOperator{\triv}{{triv}}
\DeclareMathOperator{\sgn}{{sgn}}

\newcommand{\ZZ}{{\mathbb Z}}

\newcommand{\FF} {{\mathbb F}}

\newcommand{\PP}{{\mathbb P}}
\newcommand{\CC}{{\mathbb C}}
\newcommand{\QQ}{{\mathbb {Q}}}
\newcommand{\RR}{{\mathbb {R}}}

\newcommand{\M}{{\mathfrak{m}}}

\renewcommand{\hat}{\widehat}

\newcommand{\sF}{{\mathcal F}}

\newcommand{\ssy}{\lambda}

\newcommand{\cssy}{{C_{\ssy}}}

\newcommand{\irr} {{\rm Irr}}

\newcommand \nmx{{M_m(X)}}

\newcommand \trho{{\tilde{\rho}}}
\newcommand \Rad{{\rm Rad}}

\newcommand{\beql}[1]{\begin{equation}\label{#1}}
\newcommand{\eeq}{\end{equation}}

\begin{document}



\title{A Family of Measures on Symmetric Groups  and the Field with One Element}

\author{Jeffrey C.  Lagarias}
\address{Dept. of Mathematics\\
University of Michigan \\
Ann Arbor, MI 48109-1043\\
}
\email{lagarias@umich.edu}


\thanks{The author's work  was partially supported by NSF grant DMS-1401224.}
 \date{Dec. 12,  2015}

\begin{abstract} 
For each $n \ge 1$ this  paper considers a one-parameter family of  complex-valued measures
on the symmetric group $S_n$, depending
on a complex parameter $z$.
For parameter values $z=q= p^f$ 
this measure  describes splitting probabilities of monic degree n polynomials over $\FF_q[X]$,
conditioned on being square-free. It studies these measures in  the case 
$z=1$, and shows that they have an interesting internal
structure having a representation theoretic interpretation.
These measures may encode data  relevant  to the hypothetical ``field with one element $\FF_1$".
It  additionally studies the case $z=-1$, which also has a representation theoretic interpretation.

\end{abstract}

\maketitle


\section{Introduction}\label{sec1}

This paper considers  a one-parameter family of complex-valued  measures  on
the symmetric group $S_n$, called {\em $z$-splitting measures}, introduced
by the author and B. L. Weiss in  \cite{Lagarias-W:2014}. The parameter $z$
may take complex values.  
 These  measures were constructed to
interpolate  at parameter values $z= q=p^f$, a prime power,  probability measures that
give the probabilities of given factorization type
of  monic degree $n$ polynomials over finite fields $\FF_q$, conditioned to have a
square-free factorization.  In  \cite{Lagarias-W:2014} these measures at $z=p$
arose as  limiting distributions 
 on how the prime ideal $(p)$ in $\ZZ$ splits in
the number field generated by a root of a random 
degree $n$ polynomial $f(X)= X^n +a_{n-1}X^{n-1} + \cdots + a_0 \in \ZZ[X]$ with coefficients drawn from a box $|a_i| \le B$,
as $B \to \infty$, after conditioning  on the polynomial  discriminant $D_f$ being relatively prime
to $p$. With limiting probability $1$ as $B \to \infty$ such a polynomial $f(X)$  is irreducible and has splitting
field having Galois group $S_n$, in which case
adjoining a single  root of it yields   an {\em $S_n$-extension,}
meaning a   non-Galois degree $n$ extension of $\QQ$ whose
Galois closure has Galois group $S_n$. 
The resulting splitting distributions
were compared to those in a conjecture of Bhargava \cite{Bhargava:2007}
for the distribution of
splitting types of a fixed prime ideal $(p)$  in those $S_n$-extensions $k$ of $\QQ$ 
having  field discriminant $|D_k|$ at most $ D$, in the limit   $D \to \infty$.
The  Bhargava distribution matches the $z \to \infty$ limit of the $z$-splitting measures,
which is the uniform distribution on $S_n$.
 The $z$-splitting measures for $z=p$ are also relevant to the distribution  to splitting types
of monic polynomials with $p$-adic integer coefficients  studied in Weiss \cite{Weiss:2013}.

To define the $z$-splitting measures, we first
specify them to be  constant on conjugacy classes $\cssy$ of $S_n$, which 
we label by partitions $\ssy$ specifying the (common) cycle structure of all elements
$g \in C_{\lambda}$.
For  each  $m \ge 1$  we  define the $m$-th
{\em necklace 
 polynomial} $\nmx$  by
$$
\nmx := \frac{1}{m} \sum_{d|m} \mu(d) X^{m/d},
$$
where $\mu(d)$ is the M\"{o}bius function. 
We next introduce the 
{\em cycle polynomial} 
$N_{\lambda}(X)$  attached to a partition $\ssy$, by
$$
N_{\lambda}(X) := \prod_{j=1}^n \binom{M_j(z)}{m_j(\ssy)},
$$
in which $m_j=m_j(\ssy)$ counts the number of cycles
 in $g \in S_n$ of length $j$, and for a complex number $z$ we interpret 
 $\binom{z}{k} := \frac{(z)_k}{ k!}= \frac{z(z-1) \cdots (z-k+1)}{k!}$.
 The  {\em $z$-splitting measure} $\nu_{n,z}^{\ast}$   is then 
 defined on conjugacy classes $\cssy$ of $S_n$ by
 \begin{equation}\label{113a}
\nu_{n, z}^{\ast}(\cssy) := 
\frac{1}{z^{n-1}(z-1)} N_{\ssy}(X).
\end{equation}
 The value of the measure on a single element  $g \in S_n$ with $g \in \cssy$ is
 $\nu_{n, z}^{\ast}(g) := \frac{1}{|\cssy|} \nu_{n,z}^{\ast}(\cssy)$.
  In \cite{Lagarias-W:2014} it was shown that for all integers $k \ne 0, 1$
 the measures $\nu_{n,k}^{\ast}$ are nonnegative, so are probability
 measures.  In addition a limit measure as $z \to \infty$ exists and is 
the uniform measure on $S_n$.

 This paper studies  these measures  at the parameter value  $z= 1$, which is
 the sole remaining integer value where the $z$-splitting  measure is well-defined,
 cf.  Lemma \ref{le24b}. (The formulas diverge at $z=0$.)
 We call $\nu_{n, 1}^{\ast}$ the {\em $1$-splitting measure}.
 The  $1$-splitting measure  turns out to be a  signed measure for $n \ge 3$;
  the occurrence of group elements having negative measure is a
  distinctive special feature of the  value $z=1$
  among nonzero integer values of $z$. 
  
  We show that the $1$-splitting measures for varying $n$
  possess an internal structure which respects the multiplicative
  structure of integers.  We also show  that  these signed measures for
  fixed $n$ have an  interpretation in terms of  the representation theory of $S_n$.
  This  interpretation, which is
  not apparent at values  $z=q=p^f >1$, is a  main observation of this paper. 
  
   We  study  additionally the measures at $z=-1$,
  which are nonnegative measures  having a very simple form,
  and observe that they also have a representation theory interpretation.


\subsection{Results}\label{sec11}

For the parameter value $z=1$ we show the following results. 

\begin{enumerate}
\item[(1)]
For $n \ge 2$, the $1$-splitting measure $\nu_{n,1}^{\ast}$ is supported on the conjugacy classes of $S_n$
whose associated partitions are rectangles $[b^a]$ with $ab=n$ or  are
rectangles plus a single extra box, those of type $[d^c, 1]$ with $cd=n-1$ (Theorem \ref{th31}).
These are exactly the {\em Springer regular elements} of the Coxeter group $S_n$, in the sense 
of \cite{Springer:1974}, see also \cite[Sect. 8]{RSW:2004} and \cite{RSW:2006}.
 It is a signed measure for $n \ge 3$, having total (signed) mass $1$.
\item[(2)]
The $1$-splitting measure $\nu_{n,1}^{\ast}$ 
can be  uniquely written as a sum of two (signed) measures
$$\nu_{n, 1}^{\ast} = \omega_n + \omega_{n-1}^{\ast},$$
 in which the measure $\omega_n$ is supported on partitions of type $[b^a]$, and 
$\omega_{n-1}^{\ast}$ is supported on partitions of type $[d^c, 1]$,
such that  the value of   $\omega_{n-1}^{\ast}$ summed over  the conjugacy class  $C_{[d^c,1]}$ agrees  with
the measure  $\omega_{n-1}$ on $S_{n-1}$ 
summed over the conjugacy class $C_{[d^c]}$. Thus the family of measures $\{\nu_{n,1}^{\ast} : n \ge 1\}$ are
in effect built up out of the family of measures $\{\omega_{n}: \, n \ge 1\}.$
The two measures  $\omega_n$ and  $\omega_{n-1}^{\ast}$ overlap on the 
identity conjugacy class $[1^n]$, and the (signed) mass there must be properly subdivided
between the two measures (Theorem \ref{th31}).
\item[(3)] 
The measures $\omega_n$ are computed explicitly,
and are positive measures for odd $n$ and  
(strictly) signed measures for even $n$ (Theorem \ref{th32}).

\item[(4)]
The measures $\omega_n$  respect the multiplicative  structure of integers,
in the following sense:  If $n$ has prime factorization
$$n = \prod_{i} p_i^{e_i}$$
and also factors as $n=ab$ (allowing $a=1$ or $b=1$) then
$$
\omega_n( C_{[b^a]}) = \prod_{i} \omega_{{p_i}^{e_i}}( C_{[(b_i)^{a_i}]}),
$$
in which $b_i=p_i^{e_{i, 2}}$ (and $a_i= p_i^{e_{i,1}}$) represent the maximal power of $p_i$ dividing
$b$ (resp. $a$).  In this factorization
only the prime $p_i=2$ contributes signed terms  to
the measure value, all other terms appearing are nonnegative (Theorem \ref{th33}).
\item[(5)]
For odd $n= 2m+1$ the measure $\omega_n$ is a positive measure of total mass $1$.
For even $n=2m$ the measure $\omega_{2m}$ is a signed measure of total mass $0$,
and its absolute value  measure $|\omega_{2m}|$ has total mass $1$ (Theorem \ref{thm34}).

\item[(6)]
There is a probabilistic sampling construction of  the probability  distributions $|\omega_n|$ for all $n \ge 2$ (Theorem \ref{th41}).
There is a probabilistic sampling construction (adding signs) for  the signed distributions $\omega_{2m}$
for even integers $2m$ (Theorem \ref{th42}). 
\item[(7)]
The scaled measures $n! |\omega_n|$ and $n! \omega_n$ (these measures are identical for odd $n$)
 and the signed scaled measure $(-1)^n n! \omega_{n-1}^{*}$ take integer values on group elements. They have the
 following representation-theoretic interpretations:

$(i)$ For $n \ge 1$  the class function  $n! |\omega_n|$ is  the character of a representation, with the  
representation being the representation induced from
the trivial representation $\chi_{\triv}$ on any cyclic subgroup of $S_n$ that is generated by an $n$-cycle (Theorem \ref{th51}). 

$(ii)$ For  even $n=2m$  the  class function $- (2m)!\omega_{2m}$ is  the character of a 
representation, with the representation being that 
 induced from the sign character representation $\chi_{sgn}$ on a
cyclic subgroup of $S_n$ generated by a $2m$-cycle (Theorem \ref{th52}). 

$(iii)$ For $n \ge 2$ the signed class function  $(-1)^n n! \omega_{n-1}^{*}$  is the character of a 
representation of $S_n$, with the  representation being that  induced from the $1$-dimensional
representation $ (\chi_{\sgn})^n$ on 
a cyclic subgroup of $S_{n}$ 
generated by an  $(n-1)$-cycle  holding  the symbol $n$ fixed (Theorem \ref{th54}).
\end{enumerate}

 For  the parameter value $z=-1$ 
 we  show the following  results.
 \begin{enumerate}
 \item[(8)]
 For $n \ge 2$, the  $(-1)$-splitting measure $\nu_{n, -1}^{\ast}$ is a nonnegative measure supported on the conjugacy classes
 whose associated partitions have the form $[1^n]$, the identity element, and 
 $ [2 , 1^{n-2}]$, the class  of all $2$-cycles. 
 It assigns mass $\frac{1}{2}$ to each of these conjugacy classes (Theorem \ref{th61}).
\item[(9)]
The  scaled $(-1)$-splitting measure $n! \nu_{n, -1}^{\ast}$ is the character of a 
representation $\rho_n$ of $S_n$, with the corresponding representation being 
a permutation representation. It  is the representation induced from the trivial representation
acting on the
 subgroup $H= \{ e, (12) \} \subset S_n$ given by a $2$-cycle. (Theorem \ref{th62}).
\end{enumerate}


\subsection{Field with one element}
\label{sec12}


The concept  of a  (hypothetical ) ``field with one element $\FF_1$" was suggested 
in 1957 by Tits \cite{Tits:1957}
as a way to describe the uniformity of   certain phenomena on finite geometries
associated to algebraic groups coordinatized by  points  over a finite field $\FF_q$. His theory of buildings
related algebraic groups to certain simplicial complexes. For a Chevalley group scheme $G$
he said that the Weyl group of $G$ can be viewed as the group of points of $G$ over the ``field
with one element." More generally one may consider  finite incidence geometries over  finite fields $\FF_q$
(compare \cite{Cohn:2004}) 
where  there again  may be interesting degenerate geometric objects associated to ``the field with one element." 
There has  been much recent work on developing notions of generalized algebraic geometry ``over $\FF_1$",
of which we may mention \cite{Connes-C:2010}, \cite{Connes-C:2011}, \cite{Connes-C-M:2009},
\cite{Deitmar:2005}, \cite{GHU:2011}, \cite{Lopez-L:2011}, \cite{Lorscheid:2011}, \cite{Lorscheid:2012}, \cite{Soule:2004},
\cite{Szczesny:2012},  \cite{Toen:2009}. For connections with motives, see \cite[Appendix]{RPM:2011}.

Another viewpoint on the  ``field of one element $\FF_1$"  is  purely numerical.  It  associates to certain 
  algebraic varieties defined over $\QQ$ (or $\ZZ$)
arithmetic  statistics obtained by counting points under reduction 
modulo $p$ for varying primes $p$, e.g. counting points  on the variety over $\FF_p$.
In favorable circumstances these statistics may have the feature of being interpolatable
by a  rational function $R(z)$ in a parameter $z$ which for $z= q=p^f$ (a prime power) interpolates the statistics
of the geometric object. ln \cite{Lagarias-W:2014} we termed this property  the {\em rational function interpolation
property.} Whenever this property holds one may insert the value $z=1$ and define
the resulting value $R(1)$ to be the  analogue statistic over ``the field with one element $\FF_1$".
Only a restricted class of algebraic varieties yield statistics having the rational function
interpolation property. The rational function interpolation property is known to hold for counting points over
$\FF_q$ on nonsingular toric varieties defined over $\QQ$, compare \cite{Lorscheid:2011}, \cite{Lorscheid:2012}.
One can more generally  evaluate  statistics associated 
with vector bundles and cohomology of local systems for
such varieties, viewed over finite fields $\FF_q$, and admit them as supplying data ``over $\FF_1$" when they have the 
rational function interpolation property.

The measures  on $S_n$ in this paper arose in connection with splitting
problems for polynomials defined over finite fields $\FF_q$,  studied  in  \cite{Weiss:2013}
and  \cite{Lagarias-W:2014}. This connection at parameter values 
$z=q=p^f$ is recalled  in Proposition \ref{pr26}.
The splitting distributions, conditioned on
the polynomial factorization being square-free,  have the rational function interpolation property 
in $z$ exhibited in the explicit formula
 \eqref{113a}. Therefore the measures at the parameter value $z=1$  might  be viewed 
as  statistics  associated to a geometric object over the 
``field with one element $\FF_1$''.  

 An interesting feature of the polynomial splitting interpretation is that
 the probability of a random degree $n$ polynomial over $\FF_q$ being square-free
is exactly $1- \frac{1}{q}$, independent of its degree $n$, provided $n \ge 2$.
(For degree $n=1$ this probability is $1$, independent of $q$.)  
Choosing  the parameter $q=1$ yields 
for degree $n \ge 2$ the 
probability $0$ of being square-free.  However the $z$-splitting  measures in this paper 
 compute conditional  probabilities, normalized
to specify total mass $1$  on square-free factorizations, 
which have a nontrivial limit as $z \to 1$. 
Perhaps the conditional probability aspect of  this limiting process indicates
that the  measures studied should be viewed as describing
geometric properties  associated  to  the ``$\FF_1$-tangent space" rather than to geometry over ``the  field $\FF_1$."

In terms of geometry, the $z$-splitting measure on $S_n$ at $z=q$ is associated with properties of 
the $\FF_q$-points of the open (noncompact) variety $Z_n := \PP^n \backslash \{ H_n, L_n\}$ 
where the projective space $\PP^n$ is identified with the coefficients $(a_0, a_1, ..., a_n)$ associated to the 
polynomial $f(X) = a_0 X^n + a_{1} X^{n-1} + \cdots+ a_{n-1}X+a_n$, with
$L_n := \{ a_0=0\}$ being the hyperplane ``at infinity" and  the {\em discriminant locus}
$H_n := \{ \Disc(f(X))=0\}$ being the hypersurface cut out by  the discriminant of $f(X)$,
given by a homogeneous polynomial of degree $n$. The condition $\Disc(f(X))\ne 0$
specifies that $f(X)$ has a square-free factorization.
(Removing $L_n$ takes away monic polynomials of lower degrees.)
The symmetric group $S_n$ acts on the 
roots of $f(X)$ and on the possible factorizations of $f(X)$, and the
$\FF_q$-points are distinguished by the $S_n$-action. 
The $S_n$-action distinguishing factorizations is  associated to a configuration space $\Conf(n)$ of $n$ distinct points
$(z_1, ..., z_n)$ subject to the distinctness constraint $z_i \ne  z_j$,
with associated polynomial $f(X) =  \prod_{i=1}^n (X- z_i)$.
 Church, Ellenberg and Farb \cite{CEF:2013}, \cite{CEF:2013b}.
study  aspects of the $S_n$-action on homology of this variety and its numbers of
$\FF_q$-points. Their method extracts asymptotic information as $n \to \infty$ of
various statistics on the occurrence of certain  families of representations.
 Such statistics have  topological content, and one may ask
 whether topological information on the structures above is encoded in
 the $1$-splitting measure.   The limit $q=1$ represents a different limit for the statistics
 than that studied by Church, Ellenberg and Farb. 

A further indication that  the $1$-splitting measure may have an interesting
``geometry over $\FF_1$"  interpretation  is the observation that it  is supported on  the Springer
regular elements of $S_n$.
The fact that the limit measure
is signed suggests  that the associated geometry will be  different in some aspect
from $\FF_1$-type statistics associated  with counting points
on closed varieties twisted by local systems; as indicated above
it seems to be associated to an open variety.
A geometric interpretation  would  align the observations made here 
with various known geometric and topological ``field of one element" constructions.


\subsection{Contents of the Paper }\label{sec13}

Section \ref{sec2} defines the $z$-splitting measures, and reviews basic
facts about them, mainly following \cite{Lagarias-W:2014}.
Section \ref{sec4} describes properties of the $1$-splitting measures. It splits them into a sum
of two simpler measures $\omega_n$ and $\omega_{n-1}^{\ast}$.
Section \ref{sec43} gives a probabilistic interpretation of $\omega_n$ and $|\omega_n|$.
Section \ref{sec5} gives a representation-theoretic interpretation of $\omega_n$ and $\omega_{n-1}^{\ast}$.
Section \ref{sec6} treats the case $z=-1$.
Section \ref{sec7} makes some concluding remarks.
An appendix to the paper gives tables of the  measures $\omega_n,$ $\omega_{n-1}^{\ast}$ and $\nu_{n,1}^{\ast}$ 
 evaluated on conjugacy classes of $S_n$ for $2 \le n \le 9$.


\subsection{Notation} \label{sec14}
 
 (1) $q=p^f$  denotes a power of a prime $p$, and $f=1$ is allowed.
 
 (2)  Macdonald \cite{Macdonald:1995}
 and Stanley \cite{Stanley:1997}, \cite{Stanley:1999} use $\ssy$ to denote 
  partitions of $n$ (with $n$ unspecified), writing $\ssy=(\lambda_1, \lambda_2, \cdots ,\lambda_k)$
 with  $\lambda_1 \ge \lambda_2 \ge ...\ge \lambda_k$,  alternatively writing 
 $\ssy = (1^{m_1} 2^{m_2} \cdots n^{m_n})$, with  $m_i :=m_i(\lambda)$ being the number of parts $i$  in $\lambda$.
 This paper often denotes  $\ssy= [ n^{m_n},  \cdots, 2^{m_2}, 1^{m_1}]$,
 in decreasing order,  indicating only  those parts having $m_i \ge 1$, cf.   Section \ref{sec11} 
 and the Appendix.  This notation differs from  \cite{Lagarias-W:2014}, which used 
 $\mu$ for partitions and     $c_i(\mu) :=|\{j: \mu_i \ge i\}|$ for number of parts $i$ in $\mu$.
 We view  Ferrer's diagrams of $\lambda$ drawn in British notation with the largest part at the
 top of  the diagram. The Ferrers diagram of the rectangular partition $[b^a]$ is  an $a \times b$ matrix.

 (3) For complex-valued functions $f$ on a group $G$, given any subset
$Y \subseteq G$,  we assign the value $f(Y) := \sum_{g \in Y} f(g).$
 The absolute value function $|f|: G \to \RR$ is defined by $|f|(g) := |f(g)|.$
 In particular we apply these conventions when $f$ is a character on $G$.

\section{Splitting Measures}\label{sec2}

We define  the $z$-splitting measure $\nu_{n, z}^{\ast}$ on $S_n$,
which depends on a complex parameter $z$, in Section \ref{sec23a}.
These measures are constructed using
necklace polynomials $\nmx$ and cycle polynomials $N_{\ssy}(X)$,
and we review their basic properties.

\subsection{Necklace Polynomials}\label{sec21a}

For  each degree $m \ge 1$  we first define the $m$-th
{\em necklace 
 polynomial} $\nmx$  by
\begin{equation}\label{211}
\nmx := \frac{1}{m} \sum_{d|m} \mu(d) X^{m/d}.
\end{equation}
where $\mu(d)$ is the M\"{o}bius function. One has
$M_1(X) = X, M_2(X) = \frac{1}{2}(X^2-X)$, $M_3(X) = \frac{1}{3} (X^3 -X)$
and $M_6(X) = \frac{1}{6}(X^6 - X^3 -X^2 + X)$.
The polynomial $\nmx$ has rational coefficients but takes
integer values at integers. 
For a positive integer $n$  the (positive integer) value $M_{m}(n)$ 
has  a combinatorial interpretation as counting
 the number of different necklaces having $n$ distinct colored beads taking at most $n$ colors,
 which have the property of being {\em primitive} in the sense that their
 cyclic rotations are distinct, as 
noted in 1872 by Moreau \cite{Moreau:1872}.
They were  named necklace polynomials in Metropolis and Rota \cite{Metropolis:1983}. 

The necklace polynomials at values $z= q= p^f$ count 
monic degree $n$ irreducible
polynomials over finite fields $\FF_q$.
%
%
\begin{prop}\label{pr21}
Fix a prime $p \ge 2$, and let $q= p^f$. 
For each $n \ge 1$ consider the set $\sF_{n, q}$ 
of all monic  degree $n$ polynomials 
$$
f(X) = X^n + a_{n-1} X^{n-1} + \cdots + a_1 X + a_0 \in \FF_q[X].
$$
Let $N_{n}^{irred}(\FF_q)$ count the
number of irreducible 
polynomials in $\sF_{n, q}$. Then 
$$
 N_n^{irred}(\FF_q) = M_n(q),
$$
where $M_n(X)$ is the $n$-th necklace polynomial.
\end{prop}

\begin{proof}
This very well-known formula goes back to Gauss, 
as discussed in  Frei \cite{Frei:2007}.
A proof appears  in  Rosen \cite[p. 13]{Rosen:2002}.
\end{proof}

The following result gives necklace polynomial values at $X=1$ and $X=-1$.
%
%
\begin{lemma} \label{le22}
(1) The necklace polynomial 
$\nmx$ has  $M_m(0)=0$  for $m \ge 1$ and
$$
M_m(1) = \left\{ 
\begin{array}{cl}
1 & ~~\mbox{for}~~m=1,\\
~&~\\
0 & ~~\mbox{for}~~m\ge 2.
\end{array}
\right.
$$
In addition  $(X-1)^2 \nmid M_m(X)$ for all $m \ge 2$.

(2) One has 
$$
M_m(-1) = \left\{ 
\begin{array}{cl}
-1 & ~~\mbox{for}~~m=1,\\
~&~\\
1 & ~~\mbox{for}~~m = 2,\\
~& ~\\
0 & ~~\mbox{for}~~m\ge 3.
\end{array}
\right.
$$
 \end{lemma}

\begin{proof}
(1) This well-known result is given in \cite[Lemma 4.2]{Lagarias-W:2014}. 

(2) The interesting case is $m \ge 3$.
We  let 
$\Rad(m)= \prod_{p|m} p$
 denote the {\em radical} of $m$, which is the largest square-free integer dividing $m$
and has $\Rad(m) >1$ for  $m \ge 2$, and treat three cases.
If $m$ is odd then
$$
M_m(-1) = \frac{1}{m} \sum_{k | \Rad(m)} \mu(k) (-1)^{\frac{m}{k}}= - \frac{1}{m} \sum_{k |\Rad(m)} \mu(k) = 0.
$$
Suppose $4$ divides $m$. Then $\frac{m}{k}$ is even if $k$ is square-free, hence
$$
M_m(-1)= \frac{1}{m} \sum_{k | \Rad(m)} \mu(k) (-1)^{\frac{m}{k}} =  \frac{1}{m} \sum_{k |\Rad(m)} \mu(k) = 0.
$$
 The remaining case is  $m= 2m_1$ with $m_1$ odd and $m_1 \ge 3$. Then 
\begin{eqnarray*}
M_m (-1)  &=&  \frac{1}{m} \Big( \sum_{k | \Rad(m_1)} \mu(k) (-1)^{\frac{m}{k}} + \sum_{k |\Rad(m_1)} \mu(2 k) (-1)^{\frac{m_1}{k}}  \Big) \\
&=& \frac{2}{m} \sum_{k | \Rad(m_1)} \mu(k) (-1)^{\frac{m}{k}}=0,
\end{eqnarray*}
where we used $\mu(2k) = \mu(2) \mu(k) = - \mu(k)$ for $k$ odd.
\end{proof} 

\subsection{Cycle Polynomials}\label{sec22a}

For each partition $\ssy$ of $n$ we define the {\em cycle polynomial} $N_{\ssy}(X) \in \QQ[X]$,
given by
\beql{221}
N_{\ssy}( X) := \prod_{j=1}^n \Big( {{M_j(X)}\atop{m_j(\ssy)}} \Big)
\eeq
It is  a polynomial of degree $n$ 
since $\sum_{j=1}^n j m_j= n$.  

Cycle polynomials arise as polynomials interpolating at $X=q$
the number of monic  degree $n$ polynomials over $\FF_q$
that have a square-free factorization into irreducible polynomials  
of degree type  $\lambda$. \medskip

%
%
\begin{prop}\label{pr23}
Fix a prime $p \ge 2$, and let $q= p^f$. Let $\sF_{n, q}$ denote the set
of all monic  degree $n$ polynomials with coefficients in $\FF_q$,
which has cardinality $|\sF_{n, q}|= q^n$. Then:

(1) Exactly $q^n- q^{n-1}$  polynomials in $\sF_{n,q}$
are square-free when factored into irreducible factors
over $\FF_q[X]$. Equivalently, the  probability of a 
uniformly drawn random polynomial  in $\sF_{n,q}$  hitting the discriminant locus
$\Disc(f(X))=0$ is  exactly $\frac{1}{q}$.

 (2) Let $N_{\ssy}^{*}(q)$ count the number of $f(x) \in \sF_{n, q}$ whose factorization over $\FF_q$ into
 irreducible factors is square-free 
 with factors having  degree type $\ssy:=(\ssy_1,..., \ssy_r)$,
 with  $\ssy_1 \ge \ssy_2 \cdots \ge \ssy_r$ with $\sum \ssy_i =n$, 
 having $m_j= m_j(\ssy)$ factors of degree $j$. Then 
 \beql{221b}
 N_{\ssy}^{*}( q) = N_{\lambda}(q) :=  \prod_{j=1}^n \Big( {{M_{j}(q)}\atop{m_j(\ssy)}} \Big),
 \eeq
 with  $N_{\ssy}(q)$ being the cycle polynomial $N_{\ssy}(X)$ evaluated at $X=q$.
 \end{prop}

\begin{proof}
(1) This result follows  from \cite[Prop. 2.3]{Rosen:2002}.
Another proof, due to M. Zieve, is given in \cite[Lemma 4.1]{Weiss:2013}.

(2) This equality of $N_{\ssy}^{*}(q)$ to this product is well-known, see for example
S. R. Cohen \cite[p. 256]{Cohen:1970}.
It  follows from counting all unique factorizations of the given type.
 \end{proof}

 Cycle polynomials have the following properties.
 
%
%
\begin{lemma} \label{le24} 
Let $n \ge 2$. For any  partition $\ssy$ of  $n$  
the cycle polynomial $N_{\ssy}(X)$ has the following properties:

(1) The polynomial 
$N_{\ssy}(X) \in \frac{1}{n!}\ZZ[X]$  is integer-valued.

(2) The polynomial $N_{\ssy}(X)$ has lead
term
$$
\left(\prod_{j=1}^n \frac{1}{ j^{ m_j(\ssy)} m_j(\ssy)!} \right) X^n = \frac{ |C_{\ssy}|}{n!} X^n.
$$

(3) The polynomial $N_{\ssy}(X)$ is divisible  by $X^{m}$,
where $m \ge 1$
counts  the number
of distinct cycle lengths appearing in $\ssy$.
\end{lemma}

\begin{proof}
These properties are proved in \cite[Lemma 4.3]{Lagarias-W:2014}.
\end{proof}

The following results on divisibility  of the cycle polynomial $N_{\ssy}(X)$ by powers of $X-1$ 
is the source of the  ``Springer regular element" property \cite[Sect. 5.1]{Springer:1974}.

%
%
\begin{lemma} \label{le24b} 

(1) Let $n \ge 2$. For any  partition $\ssy$ of  $n$  
the cycle polynomial $N_{\ssy}(X)$ 
 is divisible by $X-1$. 

(2) Such a polynomial
is divisible by $(X-1)^2$ if and only if 
the partition $\ssy$ has at least two distinct parts $\ssy_i > \ssy_j \ge 2$
or else has a part  $\ssy_i >1$ and at least two parts equal to $1$.
\end{lemma}

\begin{proof}
Lemma  \ref{le22}(1)  says  that $M_m(X)$ for
$m \ge 2$ contains  a factor of $X-1$. In this case $(X-1) | \Big( {{M_{m}(X)}\atop{k}} \Big)$
for any $k \ge 1$.
The only $N_{\ssy}$ not
covered by this result are those with  $\ssy = [1^n]$.

For the ``if" direction of (2) for    $n \ge 2$ one has  $(X-1) |  N_{[1^n]}=\Big( {{M_{1}(X)}\atop{n}} \Big)= \frac{X(X-1) \cdots (X-n+1)}{n!}.$
The gives a sufficient condition for $(X-1)$ to divide two different factors  in the product \eqref{221b}
defining $N_{\ssy}(X)$.  For the ``only if" part of (2) we see that the remaining partitions either have the form $ [b^a]$ with $ab=n$ or else
have the form $[d^c, 1]$ with $cd= n-1$. We must show $(X-1) || N_{\ssy}(X)$ in these cases. For the case $[b^a]$ we have 
$M_{m}(1) =0$, and we have $M_m^{'}(1) \ne 0$
by Lemma \ref{le22} (1). Furthermore $M_m(X) -j$ for $j  \ge 0$ has nonzero value $j$ at $X=1$, so contributes no extra root. 
So the multiplicity of the factor $(X-1)$ is $1$ in this case. In the remaining case  $[d^c, 1]$ the same argument applies, with the extra factor
$M_{1}(X) =X$ being nonzero at $X=1$.
\end{proof}

There is a simple formula giving the sum of all polynomials $N_{\ssy}(X)$ over 
all partitions $\lambda$ of $n$. 

%
%
\begin{lemma} \label{le25} 
 For fixed $n \ge  2$ there holds
\begin{equation}\label{eq351}
\sum_{\ssy \vdash n} N_{\ssy}(X) = X^{n-1}(X-1).
\end{equation}
\end{lemma}

\begin{proof}
Both sides are polynomials of degree $n$, 
so it suffices to check  that their values agree  at $n+2$ points.
One checks that their values agree  at $X=p^f$ for all prime powers $p^f$
using the two parts of Proposition \ref{pr23}. 
 \end{proof}

\subsection{$z$-splitting measures }\label{sec23a}

\begin{defi}\label{z-split}
{\em 
(1) The  {\em $z$-splitting measure} $\nu_{n,z}$   is 
 defined on conjugacy classes $\cssy$ of $S_n$ by
 \begin{equation}\label{231a}
\nu_{n, z}^{\ast}(\cssy) := 
\frac{1}{z^{n-1}(z-1)} \prod_{j=1}^n \binom{M_j(z)}{m_j(\ssy)}.
\end{equation}
in which  $m_j=m_j(\ssy)$ counts the number of cycles
 in $g \in S_n$ of length $j$, and for a complex number $z$ we interpret 
 $\binom{z}{k} := \frac{(z)_k}{ k!}= \frac{z(z-1) \cdots (z-k+1)}{k!}$.
 
 (2) The  measure is  extended from conjugacy classes to  elements $g \in S_n$ 
 by requiring it to be constant within a conjugacy class.
 }
 \end{defi}

 A well-known formula  for  the size of a conjugacy class states (\cite[Prop. 1.3.2]{Stanley:1997}),
\begin{equation}\label{230b}
|\cssy| = n!  \prod_{j=1}^n \,  \frac{j^{-m_j(\ssy)}}{m_j(\ssy)!}.
\end{equation}
Using it we  obtain for each $g \in C_{\lambda}$, 
 \begin{equation}\label{231}
 \nu_{n, z}^{\ast}(g) :=
  \frac{1}{n!} \cdot \frac{1}{z^{n-1}(z-1)} \prod_{j=1}^n j^{m_j(\lambda)} m_j(\lambda)! \binom{M_j( z)}{m_j(\ssy)}.
  \end{equation}
This formula shows that for each $g \in S_n$ these values are rational functions of the parameter $z$.

The $z$-splitting measure 
 on conjugacy classes of $S_n$ is written in terms of cycle polynomials $N_{\ssy}(z)$ as
\begin{equation}\label{spl-def}
\nu_{n, z}^{\ast}(C_\ssy) := \frac{1}{z^{n-1} (z-1)} N_{\ssy}(z). 
\end{equation}
 Lemma \ref{le24b} shows
 that $(z-1) | N_{\ssy}(z)$, which shows that  this measure takes well-defined values at all $z\in \CC \smallsetminus \{ 0\}$.
 Lemma \ref{le25} shows  that, as  a function of $z$,   
 \begin{equation}\label{totalsumone}
\nu_{n, z}^{\ast}(S_n) :=  \sum_{g \in S_n} \nu_{n, z}^{\ast}(g) = \sum_{\lambda \vdash n} \nu_{n, z}^{\ast}(C_\ssy) =1.
\end{equation}


\subsection{Random polynomial splitting interpretation of $z$-splitting measures}\label{sec33a}

The $z$-splitting measures at $z= q=p^f$ arise as the  splitting probabilities for factorizations of monic degree $n$
polynomials over $\FF_q$, as shown in \cite{Lagarias-W:2014}.
Recall that  $\sF_{n,q}$ denotes the set of all degree $n$ monic polynomials  $f(x) \in \FF_q[x]$.
We can factor a given $f(x)$ uniquely as 
$f(x) = \prod_{i=1}^kg_i(x)^{e_i}$, where the $e_i$ are positive integers and the 
$g_i(x)$ are distinct, monic, irreducible, and non-constant. 
We let $\ssy \vdash n$ denote the partition of
$n$ given by the degrees of the factors $g_i(x)$.

\begin{prop}\label{pr26}
Consider a  random monic polynomial $f(X)$  over the finite field $\FF_q$
drawn from the set  $\sF_{n,q}$ with the 
uniform distribution. Then the  conditional probability of $f(x)$ having a factorization 
into irreducible factors of  splitting type $\ssy$, conditioned on $g(x)$ having a square-free factorization,
is exactly $\nu_{n, q}^{\ast}({\cssy})$. That is, 
$$
\nu_{n, q}^{\ast}(\cssy) =\frac{{\rm Prob} [ f(x) ~\mbox{has splitting type} ~~\ssy ~~ \mbox{and} ~~f(x) ~\mbox{square-free}]}
{{\rm Prob} [ f(x) ~\mbox{is square-free}]}.
$$
\end{prop}

\begin{proof}
Proposition ~\ref{pr21} and Proposition \ref{pr23}\,(1)  together give
$$
\frac{{\rm Prob} [ f(x) ~\mbox{has splitting type} ~~\ssy ~~ \mbox{and} ~~f(x) ~\mbox{square-free}]}
{{\rm Prob} [ f(x) ~\mbox{is square-free}]}
= \frac{1}{q^n - q^{n-1} }
\prod_{j=1}^n \Big( {{M_j(q)}\atop{m_j(\ssy)}} \Big).
$$
Comparison of the right side  with the definition 
 (\ref{231a}) of the necklace measure  shows equality at $z=q$ with 
 $\nu_{n, q}^{\ast}(\cssy)$. 
\end{proof}

\section{Splitting Measures for $z=1$}\label{sec4}

The main object of this paper is to treat the $z$-splitting measures when  $z=1$.
The  well-definedness of the splitting measure $\nu_{n,1}^{\ast}(C_{\ssy} )$  at $z=1$
follows from the formula \eqref{spl-def} using the fact that $(X-1) | N_{\ssy}(X)$ for $n \ge 2$.
These measures turn out to be  (strictly) signed measures for all $n \ge 3$. 
These measures have total (signed) mass $1$ by \eqref{totalsumone}.


\subsection{ Decomposition into a sum of two measures attached to $n$ and $n-1$}\label{sec41}

We show that the  measure $\nu_{n,1}^{\ast}$ is supported on a small set of
conjugacy classes $\cssy$ and that it can be expressed  as a sum of two 
measures $\omega_n$ and $\omega_{n-1}^{\ast}$, at least one of which is signed, both 
constructed in terms of a family of
auxiliary measures $\{\omega_n:  n \ge 1\}$, one for each $S_n$.
The measure  $\omega_{n-1}^{\ast}$  on $S_n$
is directly  obtainable  from $\omega_{n-1}$ on $S_{n-1}$ in a simple fashion
described in the following result.

\begin{thm}\label{th31}
The signed measures  $\nu_{n, 1}^{\ast}$ have the following properties.\medskip

(1) The support of the measure $\nu_{n, 1}^{\ast}$ is exactly the
set of  conjugacy classes $[{\ssy}]$
such that $\ssy$ is one of:
\begin{enumerate}
\item[(i)]
Rectangular partitions $\ssy = [b^a]$ for $ab=n$.
\item[(ii)]
Almost-rectangular partitions $\ssy= [d^c, 1]$ for $cd=n-1$.  
\end{enumerate}

(2) 
The measure  $\nu_{n,1}^{\ast}$ is a sum of two signed measures on $S_n$, 
$$ \nu_{n, 1}^{\ast} = \omega_n + \omega_{n-1}^{\ast},$$
which are uniquely characterized for all $n \ge 1$ by the following two properties:
\begin{enumerate}
\item[(P1)]
$\omega_n$ is supported on the rectangular partitions $[b^a]$ of $S_n$,
\item[(P2)]
$\omega_{n-1}^{*}$ is supported on the almost-rectangular partitions of $S_n$,
those of the form $[d^c, 1]$, and is 
obtained from  $\omega_{n-1}$  on $S_{n-1}$, as follows. For $\lambda \vdash n$, 
$$
\omega_{n-1}^{\ast}( \cssy) := 
\begin{cases}
\omega_{n-1}(C_{\ssy^{'}}) & \text{if $\ssy= [\ssy^{'},1]$ with $\ssy^{'} \vdash n-1$,}\\
0 & \text{otherwise.}
\end{cases}
$$
\end{enumerate}
The supports of $\omega_n$ and $\omega_{n-1}^{\ast}$
overlap on the identity conjugacy class ${\ssy=[1^n]}$, viewing $[1^n]$ 
as being both rectangular and almost-rectangular.

(3) For $n \ge 2,$
$$
\nu_{n, 1}^{\ast} ( C_{[1^n] }) =  \frac{(-1)^n}{n(n-1)}.
$$

\end{thm}

\begin{proof}
(1) The support of the $1$-splitting measure $\nu_{n,1}^{\ast}(C_{\ssy}) $  
consist of all conjugacy classes  $C_{\ssy}$ for which $(X-1)^2 \nmid N_{\ssy}(X)$.
Lemma \ref{le24}(4) says that this condition holds if and only if either $\ssy= [b^a]$ with
$ab=n$ or $\ssy= [d^c, 1]$ with $cd= n-1$. 


(2) We recursively define $\omega_n(\cdot)$ in terms of $\omega_{n-1}(\cdot)$ and $\nu_{n,1}^{\ast}(\cdot)$ by
\begin{equation}\label{601}
\omega_n(\ssy) := 
\begin{cases}
\nu_{n, 1}^{\ast} (C_{\ssy}) & \text{if $\ssy = [b^a]$, with $n=ab$, $b >1$,}\\
\nu_{n, 1}^{\ast} (C_{[1^n]})- \omega_{n-1}(C_{[1^{n-1}]}) & \text{if $\ssy= [1^n],$}\\
0 & \text{otherwise.}\\
\end{cases}
\end{equation}
The initial condition  for $n=1$ is $\omega_{1}(C_{[1]}) = \nu_{1, 1}^{\ast} (C_{[1]}) = 1.$
With this recursive definition $\omega_n$ automatically satisfies property (P1), and conversely, property (P1) forces
uniqueness of this  definition on $C_{[b^a]}$ with $b >1$, and uniqueness for   the ``otherwise" term.
Next,  property (P2) requires the recursion above for $C_{[1^n]}$, which establishes that
the measure $\omega_n$ is unique if it exists. The uniqueness of $\omega_n$ then 
forces the uniqueness of $\omega_{n-1}^{\ast}$ under the condition that it sum to $\nu_{n, 1}^{\ast}$, which is 
\begin{equation}
\omega_{n-1}^{\ast}(C_{\ssy}) := \nu_{n,1}^{\ast}(C_{\ssy}) - \omega_n(C_{\ssy}).
\end{equation}
It remains to show that this recursive definition of $\omega_{n-1}^{\ast}$ is compatible with 
the already defined $\omega_{n-1}$, i.e. that it satisfies property (P2). 
By the established support condition (1) for $\nu_{n, 1}^{\ast}$, if $\omega_{n-1}^{\ast}(C_{\ssy}) \ne 0$ then necessarily
$\ssy= [d^c, 1]$ where $n-1= cd$, with $d>1$ or with $\ssy = [ 1^n]$. 
The recursive  definition above also forces
$$
\omega_{n-1}^{\ast}( C_{[1^n]}) = \omega_{n-1}( C_{[1^{n-1}]}).
$$
It remains to check that  for all partitions of the form $\ssy= [d^c, 1] \vdash n$ having $d>1$, there holds
$$
\omega_{n}^{\ast} ( C_{[d^c, 1]}) = \omega_{n-1}( C_{[d^c]}).
$$
By the recursive definition \eqref{601}, this identity is equivalent to the assertion that
$$
\nu_{n,1}^{\ast} ( C_{[d^c, 1]}) = \nu_{n-1,1}^{\ast}( C_{[d^c]}).
$$
Using the formula \eqref{spl-def}  this assertion in turn is equivalent to the assertion that for $n-1=cd$ with $d>1$, 
\begin{equation}\label{assert}
\frac{1}{t-1} N_{[d^c, 1]}(t) |_{t=1} = \frac{1}{t-1} N_{[d^c]}(t) \Big|_{t=1}.
\end{equation}
Here we have
$$
N_{[d^c, 1]} (t) = \binom{M_1(t)}{1} N_{[d^c]} (t)
$$
and the equality \eqref{assert} follows since $\binom{M_1(t)}{1} \Big|_{t=1} = 1.$ Thus property (P2) holds.

(3) For $n=1$,   $\nu_{1, 1}^{\ast} (C_{[1]}) =1$.
For  $n \ge 2$, we have
\begin{eqnarray*}
\nu_{n, 1}^{\ast} (C_{[1^n]})  &= & \frac{1}{X^n(X-1)} \prod_{i=1}^{n} \frac{(X-i+1)} {i} \, \Big|_{X=1} \\
& =&  \frac{ (-1)^{n-2} (n-2)!}{n!} = \frac{ (-1)^n}{n(n-1)}.
\end{eqnarray*}
\end{proof}


\subsection{Structure of the measures $\omega_n$}\label{sec42}

Theorem \ref{th31}  effectively reduces the study of 
the $1$-splitting  measures $\nu_{n, 1}^{\ast}$ to the  study  of the
family of  measures $\omega_{n}$,  which are signed
measures for even $n$.

\begin{thm}\label{th32}
The  measure $\omega_n$ is given
for each $n \ge 1$ and each partition $\ssy \vdash n$, as
\begin{equation}\label{omega-form}
\omega_n( \cssy) = 
\begin{cases}
(-1)^{a+1} \frac{\phi(b)}{n}  &\text{if  $\ssy= [b^a]$ for the factorization $n=ab$,}\\
0 & \text{otherwise}.
\end{cases}
\end{equation}
The  measure $\omega_n$ is supported on exactly $d(n)$ conjugacy classes,
where $d(n)$ counts the number of positive divisors of $n$. It is a nonnegative
measure for odd $n$ and is a strictly signed measure for even $n$.
\end{thm}

\begin {proof}
By definition the measure $\omega_n$ is constant on
conjugacy classes and is 
supported on elements 
having cycle structure $\ssy = [b^a]$
where $n = ab$; there are $d(n)$ such  classes.
If $b >1$ then we have 
$\omega_n(C_{[b^a]}) = \nu_{n,1}^{\ast}(C_{[b^a]})$. 
In this case Lemma \ref{le22}\,(1) gives $(X-1) | M_b(X)$ and also
$$
\frac{M_b(X)}{X-1}\Big|_{X=1}= M_b^{'}(1) =  \prod_{p|b} \Big(1- \frac{1}{p}\Big)= \frac{\phi(b)}{b}  >0,
$$
where $\phi(b)$ is Euler's totient function.
In addition for $b >1$ we have
$$
(M_b(X) - j)\Big|_{X=1} = -j.
$$
We obtain
$$
\nu_{n,1}^{\ast}(C_{[b^a]})= \frac{1}{a!} \cdot
\frac{\phi(b)}{b} \prod_{j=1}^{a-1} (-j) = (-1)^{a-1}\frac{\phi(b)}{ab}= (-1)^{a+1} \frac{\phi(b)}{n} ,
$$
where $\phi(b)$ is Euler's totient function. Thus for $b >1$ we obtain
$$
\omega_n ( C_{ [b^a]} )  = \nu_{n,1}^{\ast}(C_{[b^a]}) =(-1)^{a+1} \frac{\phi(b)}{n}.
$$

For the remaining case $b=1$, where $a=n$, we define for  $n =1$, 
$$\omega_{1}(C_{[1]}) = \nu_{1,1}^{\ast}(C_{[1]})= 1.$$
For  $n \ge 2$ we have 
 the recursion (as in Theorem \ref{th31})
$$
\omega_{n} (C_{[1^n]})= 
\nu_{n,1}^{\ast}(C_{[1^n]}) - \omega_{n}^{\ast}( C_{[1^{n-1}, 1]})=\nu_{n,1}^{\ast}(C_{[1^n]}) - \omega_{n-1}( C_{[1^{n-1}]}).
$$
Using the formula of Theorem \ref{th31} (3) we have
$$\nu_{n, 1}^{\ast} (C_{[1^n]}) = \frac{(-1)^{n}}{n(n-1)},$$
and it  follows that
$$\omega_2 (C_{[1^2]}) =\nu_{2,1}^{\ast}(C_{[1^2]}) - \omega_{1}( C_{[1]})
= \frac{1}{2} - 1 = -\frac{ 1}{2}.
$$
We now prove by induction on $n \ge 2$ that
\begin{equation}\label{case1}
\omega_n (C_{[1^n]}) = \frac{(-1)^{n+1}}{n}.
\end{equation}
This result is equivalent to the identity
$$
\frac{(-1)^{n+1}}{n} = \frac{(-1)^{n}}{n(n-1)} - \frac{(-1)^{n}}{n-1},
$$
which is
$$
\frac{1}{n-1}-\frac{1}{n} = \frac{1}{n(n-1)}.
$$
The formula \eqref{case1} matches the theorem's formula \eqref{omega-form} for $b=1$.

Finally we observe that for odd $n$ one has $(-1)^{a+1}=1$ so the measure
is nonnegative. For even $n$ one may always choose $n=ab$ with  $a=1$ and $a=2$,
so the measure is strictly signed. 
\end{proof}

We use the  explicit description of the 
measures $\omega_n$ 
to derive the following consequences about their structure.

\begin{thm}\label{th33} 
The  measures $\omega_{n}$ on $S_n$  have the following properties.

(1) The measure $|\omega_n|$ has total mass $1$, so is a probability measure.

(2) For $n=2m+1$ the  measure $\omega_{2m+1}$ is nonnegative and has total mass $1$, so is a probability measure.
It  is supported on even permutations, so  
its restriction $\omega_n|_{A_n}$ to the alternating group $A_n$ is a probability measure.

(3) For  $n=2m$ the 
measure $\omega_{2m}$ is a signed measure having total signed mass $0$.
The measure of $\omega_n$  is  nonnegative on odd permutations, and there has  total mass $\frac{1}{2}$. 
It is nonpositive on even permutations and there has  total signed mass $-\frac{1}{2}$. 
Thus the measure  $-2\omega_{2m}|_{A_{2m}}$ restricted to the alternating group $A_{2m}$ is a probability measure.

(4) The  family of all 
measures $\omega_n$ has an internal product structure compatible
with multiplication of integers. Setting $n =\prod_{i} {p_i}^{e_i}$,
and for any factorization $ab=n$, there holds
\begin{equation}\label{prodform}
\omega_n( C_{[b^a]}) = \prod_{i} \omega_{{p_i}^{e_i}}( C_{[(b_i)^{a_i}]}),
\end{equation}
in which $b_i=p_i^{e_{i, 2}}$ (resp. $a_i= p_i^{e_{i,1}}$) represent the maximal power of $p_i$ dividing
$b$ (resp. $a$), so that $e_{i,1} + e_{i, 2} = e_i$.  We allow some or all values $e_{i,j}=0$ for $j=1, 2$,
so values  $b=1$ (resp. $a_i=1$) are allowed. 
\end{thm}


\begin{proof} 
We define the signed mass 
$$
\M_n := \sum_{g \in S_n} \omega_n(g) = \sum_{\ssy \vdash m} \omega_n(C_{\ssy}).
$$
We show that $\M_{2m}=0$ and $\M_{2m+1}=1$ for all $m \ge 1$.
We  have $\M_1 =1$. 
We  have  for all $n \ge 2$ the relation
\begin{eqnarray*}
 \M_n + \M_{n-1}& = & \M_n + \sum_{\ssy' \vdash n-1} \omega_{n-1} (C_{\ssy'}) \\
 & = & \sum_{\ssy \vdash n} \omega_n(C_{ \ssy}) + \sum_{\ssy \vdash n}\omega_{n-1}^{\ast}(C_{\ssy}) \\
 & = &  \sum_{g \in S_n} \nu_{n, 1}^{\ast}(g) =1.\\
 \end{eqnarray*}
since  only elements of form $\ssy= [\ssy', 1]$ contribute in the second sum on the second line. 
The relation $\M_n + \M_{n-1} =1$ now yields  by induction on $m \ge 1$ that each $\M_{2m}=0$ and each $\M_{2m+1}=1$.

(1) Using the formula in Theorem \ref{th32} we have for any $n \ge 1$,
$$
\sum_{\ssy} |\omega_n| (\cssy) 
= \sum_{ab=n} |\omega_n|(C_{ [b^a]})=\sum_{ab=n}  \frac{\phi(b)}{n} = \frac{1}{n} \Big(\sum_{b | n} \phi(b)\Big) =1.
 $$
 Thus $|\omega_n|$ is  a probability measure for all $n \ge 1$.

 (2)  Suppose  that $n=2m+1$ is odd. Then $n=ab$ has both $a, b$ odd, so
all $\omega_n(C_{[b^a]}) >0$  in \eqref{omega-form}, and we conclude that
$\omega_{2m+1}$ is a nonnegative measure.  In
addition all permutations of cycle shape $[b^a]$ are
even permutations, so the support of $\omega_{2m+1}$ is contained in the alternating group $A_{2m+1}$.
Now $\omega_{2m+1} = |\omega_{2m+1}$, so by assertion (1) it is  a probability measure, proving (2).

(3) Suppose that $n=2m$ is even. 
A permutation $g$ of cycle type $[b^a]$ is an odd permutation if and only if the integer
$a$ is odd  and $b$ is even. This condition holds exactly when $a$ is odd, and  in this case $(-1)^{a+1}=1$. 
In consequence
$g \in C_{[b^a]}$  is an even permutation if and only if $a$ is
even, in which case  $(-1)^{a+1}=-1$. It follows  by Theorem \ref{th32} 
that the measure $\omega_{2m}$  is nonnegative on odd permutations  and is nonpositive
on even permutations. 

We showed that $\omega_{2m}$ has total mass $\M_{2m}=0$, and we also know
 $|\omega_{2m}|$ has
total mass  $1$ by (1).
It follows that  the positive elements of the measure have total mass $1/2$ and the 
negative elements have total mass $-1/2$. 
All the  negative weight elements lie in the alternating group
$A_{2m}$, and all the positive weight elements  lie in $S_{2m} \smallsetminus A_{2m}$.
The assertions about   $-2 \omega_{2m}|_{A_{2m}}$ immediately follow,
proving  (3).

(4) The factorization formula \eqref{prodform} is easily verified by direct calculation using the formula for $\omega_n$ of Theorem \ref{th32}.
 Ignoring signs, it asserts
$$
\frac{\phi(b)}{n} =  \prod_{i} \frac{\phi(p_i^{e_{i,2}})}{p_i^{e_{i}}},
$$
an identity which holds by multiplicativity of the Euler totient function.
To verify the sign condition, note that it automatically holds on both sides if $n=2m+1$, since all terms on
both sides are nonnegative by  assertion (2).
 If $n=2m$, then by assertion (3) the left side  $\omega_n(C_{[b^a]})$  of \eqref{prodform} 
 is negative only on even permutations, which
 occur only in $C_{[b]^a}$ when $a$ is even.  The sign of the right side comes  only from
 the factor attached to   $p=2$. Now Theorem  \ref{th32} applied to $\omega_{2^{e}}$ for $e \ge 1$ takes negative
 values exactly when $a$ is even, and (4) follows.
\end{proof}

Various properties of the splitting measure $\nu_{n, 1}^{\ast}$ 
follow from the  explicit formula for the 
measure $\omega_{n}$.


\begin{thm}\label{thm34} 
The  $1$-splitting measures $\nu_{n,1}^{\ast}$  has the following properties
for all $n \ge 2$.

(i) The  measure $\nu_{n, 1}^{\ast}$ is nonnegative on odd permutations $S_n \smallsetminus A_n$ and 
there has total mass $\frac{1}{2}$.
It is a signed measure on even permutations $A_{n}$ and there has total (signed) mass $\frac{1}{2}$.

(ii) The absolute value  measure $|\nu_{n,1}^{\ast}|$ has total mass $2- \frac{2}{n}$, with  mass $\frac{1}{2}$ over the set of odd permutations and mass $\frac{3}{2}- \frac{2}{n}$ over the set of even permutations. 
\end{thm}

\begin{proof}
Assertion (i) follows from Theorem \ref{th33} (2) and (3). Here the sign of $[d^c, 1]$ in $S_n$ agrees with the sign of $[d^c]$ in $S_{n-1}$.
On odd permutations   $\omega_n$ and $\omega_{n-1}$ are nonnegative,  and restricted to them
one measure  has total mass $0$ and the other has total mass $\frac{1}{2}$.
On even permutations one of them is nonnegative with total mass $1$ and the other nonpositive with total mass $-\frac{1}{2}.$

For assertion (ii)  we have by assertion (i) 
that $|\nu_{n,1}^{\ast}|$ equals  $\nu_{n,1}^{\ast}$ on odd permutations and it there
has total mass $\frac{1}{2}$. 
By  Theorem \ref{th31} the total mass of the measure $|\omega_n| + |\omega_{n-1}^{\ast}|$ equals that of
the measure $|\omega_n| + |\omega_{n-1}|$, which by Theorem \ref{th33} (1) equals $2$. 
The former  measure agrees with $|\nu_{n, 1}^{\ast}|$ away from the identity element $C_{[1^n]},$ and the  identity
element is an even permutation. Thus the mass of $|\nu_{n, 1}^{\ast}|$ on even permutations is $\frac{3}{2}$ minus
a correction from the identity element.  At the identity element one has
$$
(|\omega_n| + |\omega_{n-1}^{\ast}|)(C_{[1^n]})= \frac{1}{n}+ \frac{1}{n-1} = \frac{2}{n} + \frac{1}{n(n-1)},
$$
while $|\nu_{n,1}^{\ast}|(C_{[1^n]})= \frac{1}{n(n-1)},$ giving a correction of $-\frac{2}{n}$, verifying assertion (ii). 
\end{proof}

\section{Probabilistic characterization of  positive measures  $|\omega_n|$
and signed measures $\omega_n$ at $z=1$}\label{sec43}

In this section we show there is an  description of the absolute  probability measure $ |\omega_n|$
 as  the output of a probabilistic sampling method. Additionally we  give a  probabilistic sampling method to
draw random  elements of the signed measure  $ \omega_n$.\\

{\bf Random power of $n$-cycle distribution.}
\begin{enumerate}
\item
 Draw an $n$-cycle $g$ from $S_n$ uniformly, with probability $\frac{1}{(n-1)!}$
for each $n$-cycle. 

\item
Draw an integer  $ 1\le j \le n$ uniformly  
 with probability $\frac{1}{n}$,
independently of the draw of $g$. 

\item
Set $h = g^j$. Take the induced  distribution of $h$ on the elements of $S_n$.
\end{enumerate}

\begin{thm}\label{th41}
The absolute value 
measure $|\omega_{n}|$ on $S_n$ 
is a probability measure that 
coincides with that given by the random power of $n$-cycle distribution.
\end{thm}

\begin{proof}
We let $\omega_n^{D}$ denote
the random power of $n$-cycle distribution. Both distributions
are constant on conjugacy classes. 
It suffices to check that  the probabilities of this distribution agree with  those
inferred from  Theorem \ref{th32}, which for  $n = ab$ are
$$
|\omega_n|(C_{[b^a]}) = \frac{\phi(b)}{n},
$$
and which are $0$ on  all other conjugacy classes $C_{\lambda}$.

By hypothesis any sample element $g \in C_{[n]}$, whence
 $h = g^j $ has cycle structure  of the form $[b^a]$
where  $a= \gcd(j,n)$ and $n=ab$.Therefore the distribution $\omega^{D}$
is supported on the conjugacy classes of form $C_{[b^a]}$,
and furthermore all elements in such a conjugacy class are drawn with  the same probability.

The probability density is determined entirely by the value of $a=\gcd(j,n)$, 
which specifies both $a$ and $b$. This 
divisibility  condition factorizes over prime powers $n = \prod_{i} p_i^{e_i}$.
 For $1 \le k \le e_i$, the condition that $p_i^k$ divides a randomly drawn
 $j \in [1, n]$ is  the same  as the condition $p_i^k$ divides $a = \gcd(j, n)$, 
 and the latter occurs with probability   $\frac{1}{p_i^k}$.
Therefore the probability that $p_i^{k}$ exactly divides $\gcd(j,n)$
is $\frac{1}{p_i^k}(1- \frac{1}{p_i})$ if $k< e_i$ and is $\frac{1}{p_i^{e_{i}}},$
if $k= e_i$. On the other hand, $p_i^{e_i -k}$ exactly divides $b$ so that
this probability always equals 
$\frac{\phi(p_i^{e_i -k})}{p_i^{e_i}}$.  We deduce that
$$
|\omega_n^{D}|(C_{[b^a]}) = \frac{\phi(b)}{n}, 
$$
giving the desired equality.
\end{proof}
 
 There is  an analogous probabilistic sampling description of 
the  signed measures $\omega_n$ for $n \ge 2$.\\
 
 {\bf Signed random power of $n$-cycle distribution.}
\begin{enumerate}
\item
 Draw an $n$-cycle $g$ from $S_n$ uniformly, with probability $\frac{1}{(n-1)!}$
for each $n$-cycle. 

\item
Draw an integer  $ 1\le j \le n$ uniformly,  
independently of the draw of $g$, probability $\frac{1}{n}$. 

\item
Set $h = g^j$. Assign to $h$ its sign $\sgn(h) = (\sgn(g))^j= (-1)^{(n+1)j}$. 
Take the induced signed  distribution of $h$ on the elements of $S_n$.
\end{enumerate}

This distribution gives something new  only when $n=2m$ is even, and is the
unsigned distribution above if $n=2m+1$ is odd.

\begin{thm}\label{th42}
Let $n$ be even.
The   
measure $\omega_{n}$ on $S_n$ 
coincides with that  given by the signed random power of $n$-cycle distribution.
\end{thm}

\begin{proof}
We suppose $n=2m$ is even, and we let $\omega_n^{SD}$ denote
the signed random power of $n$-cycle distribution. 
We check  that the probabilities $\omega_n^{SD}$ agree with those
given in Theorem \ref{th32}. The cycle structure of 
 of $h=g^j$ is of the form $[b^a]$
where $n= ab$ with $a= \gcd(j, n).$  
If $a$ is even then  $\sgn(h)= (\sgn(g))^a =1$, while if $a$ is odd
then $b$ is even and $\sgn(h) = (\sgn(g))^a = \sgn(g)= -1$. It follows that 
 $$
\omega_n^{SD}(C_{[b^a]}) = (-1)^{a+1} \omega_n^{D}(C_{[b^a]}).
$$
By Theorem \ref{th32} we have 
$$
\omega_n^{D} ([b^a]) = |\omega_n|(C_{[b^a]}) = \frac{\phi(b)}{n}
$$
whence 
$$
\omega_n^{SD}(C_{[b^a]})=(-1)^{a+1}\frac{\phi(b)}{n} = \omega_n( C_{[b^a]}).
$$
as asserted. 
\end{proof}


\section{Representation-theoretic interpretation of $\omega_n$ }\label{sec5}

 The measures   $\omega_n$ (resp. $|\omega_n|$) are class functions on $S_n$, 
 so they can be viewed as rational linear combinations of irreducible  
 characters of $S_n$.
  We show that if one 
  rescales  these measures by  the factor $n!$, which is 
  the smallest positive factor arranging  that all character values become integers, then
  $n! \omega_n$ is the character of a  virtual representation;
  that is,  it is an integral linear combination of characters of irreducible representations.
  We also show also that $n!|\omega_n|$ is the character of a 
representation, i.e.  it is a nonnegative integral linear combination
of characters of irreducible representations.
 In this section  we write $\chi_{\pi}$ for
the character of a 
representation $\pi$ and we sometimes identify one-dimensional representations
with their characters, i.e. $\rho_{\triv} = \chi_{\triv}$ and $\rho_{\sgn} = \chi_{\sgn}$.



\subsection{The  measure $n! |\omega_n|$ is the character of a  representation of $S_n$}\label{sec51}


 \begin{thm}\label{th51}
 For all $n \ge 1$  the class function  $n!|\omega_n|(g)$ for $g \in S_n$ is the  character  of
 the induced representation
 $$
 \rho_n^{+} := Ind_{C_n}^{S_n}( \chi_{\triv}),
 $$
 from the cyclic group $C_n$ generated by an $n$-cycle of $S_n$, carrying the trivial representation $\chi_{\triv}$.
  The representation $\rho_n^{+}$ is  of degree $(n-1)!$ and for $n \ge 3$ is  a reducible representation.
 The trivial representation occurs in $\rho_n^{+}$  with multiplicity $1$ and the sign representation $\chi_{\sgn}$ 
 occurs with multiplicity $1$ if $n$ is odd and multiplicity $0$ if $n$ is even.
 \end{thm}
 
 \begin{proof} 
 We let $C_n = \langle h \rangle$ denote the cyclic subgroup of $S_n$ represented by the $n$-cycle
 $h= (1 2 3 \cdots n)$, and let $\chi_{\triv}$ denote the trivial
  representation.  The cycle structure of $h^k$ for $1 \le k \le n$ is $[b^a]$ where $b = n/ \gcd(n, k)$.
 In particular there are $\phi(b)$ elements of $C_n$ having cycle structure $[b^a]$, for each $b | n$.
 
 The induced representation $\rho_n^{+} := \mathrm{Ind}_{C_n}^{S_n}( \chi_{\triv})$ is a permutation representation 
 of degree $(n-1)!$. 
 We compute its  character $\psi_{n}^{+} := \psi_{\rho}$ (with $\rho= \rho_n^{+}$) using
 the Frobenius formula for the character of an induced representation  $\psi(g) = \mathrm{Tr}(  \mathrm{Ind}_{H}^{G}(\sigma)(g)),$
 in terms of the character  $\chi(h) =\mathrm{Tr}( \sigma(h))$ of the original representation $\sigma$ on $H$, cf. \cite[Sect. 3.3]{FH:1991}.
 Applied to $S_n$ it states
 \begin{equation}\label{Frobenius}
 \psi(g) = \sum_{ x \in S_n/H} \hat{\chi} (x^{-1} g x),
 \end{equation}
 in which
 \begin{equation}\label{517}
 \hat{\chi}(g) = \begin{cases} 
 \chi (g)  & \, \mbox{if} \quad g \in H \\
 0 & \, \quad \mbox{otherwise}
  \end{cases}
 \end{equation}
 We can also eliminate the cosets and write for any subgroup
 \begin{equation}\label{510}
 \psi(g) = \frac{1}{|H|}  \sum_{ x \in S_n} \hat{\chi} (x^{-1} g x),
\end{equation}
 For a conjugacy class $C_{\ssy}$ we set
$$
\psi(C_{\ssy}) = \sum_{g \in C_{\ssy}} \psi(g) = |C_{\ssy}| \psi (g).
$$
In the case $H=C_n$ and $\chi_{\triv}$ we have
 \begin{equation}\label{511}
 \psi_{n}^{+}(g) = \frac{1}{n}  \sum_{ x \in S_n} \hat{\chi}_{\triv} (x^{-1} g x),
\end{equation}
Clearly $\psi(g) =0$ if $g$ is not conjugate to some element of $C_n$,
so we have 
$$
\psi_{n}^{+}(C_{\ssy}) =0  \quad\mbox{when} \quad \ssy \ne [b^a] \quad \mbox{for any} \, ab=n.
$$  
For the exceptional case, the formula \eqref{511} counts 
$$
\psi_{n}^{+}(C_{ [b^a]}) = | \{ (g', x, h) :   h= x g' x^{-1} \, \mbox{with}  \quad g' \in C_{[b^a]} ,  h \in C_n,  x\in S_n \, \}|.
$$
There are $|C_{[b^a]}| = \frac{n!}{b^a a!}$ choices for $g$, there are $\phi(b)$ choices for $h$, and for each
such pair there are $|N(\langle g' \rangle)| = b^a a!$ choices of $x$, 
in which  $|N(G)|$ denotes the cardinality of the
normalizer of the subgroup $G$ in $S_n$. 
We obtain
$$
\psi_{n}^{+}(C_{[b^a]})= \frac{1}{n} \Big( \frac{n!}{b^a a!} \cdot b^a a! \cdot \phi(b) \Big) = n! \frac{\phi(b)}{n}.
$$
On comparing this character with the formula for the class function $n! |\omega_n|(C_{[b^a]})$
implied by Theorem \ref{th32} we find agreement
$$
\psi_{n}^{+} (C_{\lambda})=n! |\omega_n|(C_{\lambda})  \quad \mbox{for all} \,\, \ssy \vdash n.
$$

 Now that we know the class function $n!|\omega_n|$ is the character $\psi_n^{+}$ of a
 representation $\rho_n^{+}$ , we may write it in  terms of the basis of irreducible characters as 
 $$
\psi_n^{+}=  n!|\omega_n| = \sum_{ \pi \in \irr(S_n)} m(\pi; n! |\omega_n| ) \chi_{\pi}.
$$
 with nonnegative integer multiplicities $m(\pi; n!|\omega_n| )$. Using the Hermitian inner product
 on class functions 
  for any nonzero class function $f (\cdot)$, the (signed) multiplicity is the  real number
 $$
 m(\pi;  f)  :=  \langle f, \chi_{\pi} \rangle  =   \frac{1}{n!} \sum_{g \in S_n}\overline{ f(g)} \chi_{\pi}(g)
 $$
  We know by Theorem \ref{th33}
   that  the total mass  $|\omega_n|(S_n)$ is $1$ so that  
 $$
 \psi_{n^{+}}(S_n) = n! |\omega_n|(S_n) = n! \big(\sum_{g \in S_n} |\omega_n|(g) \big) = n!.
  $$
 Now the mass over $S_n$ of the trivial representation $\chi_{\triv}$ is $n!$,
  whence the multiplicity of the trivial representation in $|\omega_n|$ is 
 $$
 m(\chi_{\triv}; \rho_n^{+}) = m( \chi_{\triv}; n!|\omega_n|) = 
 \langle  n!\,|\omega_n| ,\chi_{\triv}\rangle
 = 1.
 $$
Since the  representation $\rho_n^{+}$ has degree $(n-1)!$ and a summand of degree $1$  it must be reducible for all $n \ge 3$.
 
 Now consider the sign representation $\chi_{\sgn}$. By Theorem \ref{th33}(ii)
all the positive values of  the function $(2m)!\,\omega_{2m}$ are taken  on odd permutations
and all the negative values are taken on even permutations, and the function $(2m)! |\omega_{2m}|$
has the same  mass taken over all odd permutations versus over all even permutations. Thus
$$
 m(\chi_{\sgn}; \rho_{2m}^{+}) = m( \chi_{\sgn}; (2m)! |\omega_{2m}|) 
 = \langle  -(2m)! \, \omega_{2m} , \chi_{\triv}\rangle
 = 0,
$$
since $ \langle  \omega_{2m},\chi_{\triv}\rangle = \frac{1}{(2m)!} \sum_{g \in S_{2m}}\omega_{2m}(g) =0.$
  \end{proof}


\subsection{The measure $- (2m)! \omega_{2m}$ is the character of a  representation of $S_{2m}$}\label{sec52}

 Recall  that if $n=2m+1$ is odd then
 $\omega_{2m+1}= |\omega_{2m+1}| $, which is handled by Theorem \ref{th51}.
 

 \begin{thm}\label{th52}
   If $n=2m$ is even  then the class function 
 $- n!\omega_n$ is the character of
 the induced representation
 $$
 \rho_{2m}^{-}:= Ind_{C_{2m}}^{S_{2m}}( \chi_{\sgn})
 $$
 from the cyclic group $C_{2m}$ of a $2m$-cycle in $S_{2m}$, carrying on it the sign representation $\chi_{\sgn}$.
 The  representation $\rho_{2m}^{-}$
  is of degree $(2m-1)!$ and   is a  reducible representation for $m \ge 2$. 
 The trivial representation $\chi_{\triv}$ of $S_{2m}$ occurs in $\rho_{2m}^{-}$ with  multiplicity $0$ and the sign representation $\chi_{\sgn}$ occurs
 with multiplicity $1$. 
 \end{thm}

\begin{proof} 
Let $C_{2m}$ be a cyclic subgroup of $S_{2m}$ generated by a $(2m)$-cycle.
The induced representation $\rho_n^{-}=  \mathrm{Ind}_{C_{2m}}^{S_{2m}}( \chi_{\sgn})$ is a representation 
 of degree $(n-1)!$, since  the sign character on $C_{2m}$  is a representation of degree $1$.

 We  compute the character $\psi_n^{-}$ of $\rho_n^{-}$  using
 the Frobenius formula 
 $$
 \psi_n^{-}(g)  = \sum_{ x \in S_{2m}/C_{2m}} \hat{\chi}_{\sgn} (x^{-1} g x).
 $$
The sign character is constant on every nonzero term in this sum, so that we have
$$
 \psi_n^{-}(g)  = \chi_{\sgn} (g) \Big(\sum_{ x \in S_{2m}/C_{2m}} \hat{\chi}_{\triv} (x^{-1} g x)\Big) = \chi_{\sgn}(g) \psi_n^{+}(g).
$$
Since $\psi_n^{-}(g)$ is a class function we conclude using Theorem \ref{th51} that
$$
\psi_n^{-}(C_{\ssy} )  = \chi_{\sgn}(g) \psi_{n}^{+}(C_{\ssy})=  \chi_{\sgn}(g) (2m)! |\omega_{2m}|(C_{\lambda}).
$$ 
Now for $g \in C_{[b^a]}$ we have $\chi_{\sgn}(g) = (-1)^a$
and from Theorem \ref{th32} we have 
 $|\omega_{2m}|(C_{[b^a]})=\frac{\phi(b)}{2m}$, whence 
$$
 \psi_n^{-}(C_{\ssy} ) =
\begin{cases} 
(2m)! (-1)^a \frac{\phi(b)}{2m}
  & \, \mbox{if} \quad  \ssy= [b^a] \,\, \mbox{with} \,\, ab=2m,  \\
 0 & \, \quad \mbox{otherwise}.
  \end{cases}
$$
Comparison with Theorem \ref{th32}  yields
$$
\psi_n^{-}(C_{\ssy})= -(2m)! \omega_{2m}(C_{\ssy}) ,
$$
as asserted. 

Since the total signed mass of $\omega_{2m}$  is $\M_{2m} = 0$, the multiplicity of the trivial representation in $\rho_{2m}^{-}$ is
$$
m( \chi_{\triv}; \rho_{2m}^{-})=\frac{1}{(2m)!} \sum_{g \in S_{2m}} \chi_{\triv}(g) \psi_{n}^{-}(g) = \sum_{g \in S_{2m}} -\omega_{2m}(g) =0.
$$
Next for $n =2m$ we determine the multiplicity of the sign representation to be 
\begin{eqnarray*}
m( \chi_{\sgn}; \rho_{2m}^{-}) & = & \frac{1}{(2m)!} \sum_{g \in S_{2m}} \chi_{\sgn}(g) \psi_n^{-}(g)\\
& =&
\sum_{a | 2m} \sum_{g \in C_{[b^a]}} (-1)^a (- \omega_{2m}(g))\\
&= &\sum_{g \in S_{2m}} |\omega_{2m}|(g) =1.
\end{eqnarray*}
Since $\rho_{2m}^{-}$ on $S_{2m}$ has the sign representation as a constituent, it is  reducible for all $m \ge 2.$
 \end{proof}


\subsection{Multiplicities of representations }\label{sec53}

The multiplicities of  irreducible  representations
$\pi$ in  $\rho_n^{\pm}$ exhibit some symmetries with respect to the sign character.


\begin{thm}\label{th53}
Let $\pi$ denote an irreducible representation of $S_n$, and $\psi_{\pi}$ its character, and define
$$
\rho_n^{+}:=  \mathrm{Ind}_{C_{n}}^{S_{n}}( \chi_{\triv}) \quad \mbox{and} \quad \rho_{n}^{-}:=  \mathrm{Ind}_{C_{n}}^{S_{n}}( \chi_{\sgn}).
$$
Then:

(1) If $n=2m+1$ then $\rho_{2m+1}^{+}=\rho_{2m+1}^{-}$, and the multiplicity
\beql{531}
m(\pi; \rho_{2m+1}^{+}) =  m( \pi \otimes \rho_{\sgn}; \rho_{2m+1}^{+}).
\eeq

(2) If $n=2m$   then  the multiplicity
\beql{532}
m(\pi; \rho_{2m}^{+}) =  m( \pi \otimes \rho_{\sgn}; \rho_{2m}^{-}).
\eeq
\end{thm}

\begin{proof}

(1) Suppose  $n =2m+1$.  
 By Theorem \ref{th33} (2) the character
  $\psi_{n}^{+} = (2m+1)!\omega_{2m+1} $ is supported on even permutations,
  and $\omega_{2m+1} = |\omega_{2m+1}|$. 
 If $\pi$ is any irreducible representation on $S_{2m+1}$ then 
 \begin{eqnarray*}
 m(\pi; \rho_{2m+1}^{+}) &=&  m(\pi; (2m+1)! \, \omega_{2m+1})\\
  &= &  \langle (2m+1)! \, |\omega_{2m+1}|, \chi_{\pi} \rangle\\
    & =  & \langle  (2m+1)!\, |\omega_{2m+1}|,\chi_{\pi}\chi_{\sgn} \rangle\\
& = & m(\pi \otimes \chi_{\sgn}; \rho_{2m+1}^{+} ).
\end{eqnarray*}

(2)  Suppose now  that $n=2m$. By Theorem \ref{th33}(3)   the character $\psi_{n}^{-} = -(2m)!\omega_{2m} $
   is positive on odd  permutations and negative
  on even permutations, and (up to sign) has equal mass $\frac{1}{2}(2m)!$ on each set.   If $\pi$ is any irreducible representation on $S_{2m}$ then 
  \begin{eqnarray*}
  m(\pi; \rho_{2m}^{+}) &= & \langle (2m)! \,|\omega_{2m}|,  \chi_{\pi} \rangle\\
    &=  & \langle  (2m)!\,|\omega_{2m}|\chi_{\sgn}, \chi_{\pi}\chi_{\sgn}  \rangle\\
    & = &  \langle - (2m)!\, \omega_{2m} ,  \chi_{\pi}\chi_{\sgn}\rangle\\
    & = & m(\pi \otimes \chi_{\sgn}; -(2m)! \,\omega_{2m}) = m(\pi \otimes \chi_{\sgn}; \rho_{2m}^{-} ).
\end{eqnarray*}

\end{proof}


\subsection{The measure $(-1)^{n} n! \omega_{n-1}^{\ast}$ is the character of a  representation of $S_n$}\label{sec54}

  We show that a suitably rescaled version of the measure $\omega_{n-1}^{\ast}$ in Theorem \ref{th31}
 is the character of a 
 representation of $S_n$ which is supported on conjugacy classes of shape $[d^c, 1]$ with $cd = n-1$.
 

 \begin{thm}\label{th54}
(1)  For each  $n \ge 2$  the class function  $(-1)^{n} \, n! \, \omega_{n-1}^{\ast}(g)$ for $g \in S_n$ is the  character $\psi_n^{(L)}$ of
 the induced representation
 $$
 \rho_n^{(L)} :=  \mathrm{Ind}_{C_{n-1}}^{S_n}( (\chi_{\sgn})^n ),
 $$
 with the cyclic group $C_{n-1} \subset S_n$ generated by an $(n-1)$-cycle that  holds the symbol $n$ fixed. 
  The representation $\rho_n^{(L)}$ is  of degree $n \cdot (n-2)!$ and for $n \ge 2$ is  a reducible representation.
 
 (2) The representation $\rho_n^{(L)}$   is also given as the induced representation 
 $$
 \rho_n^{(L)} =  \mathrm{Ind}_{S_{n-1}}^{S_n}( \rho_{n-1}^{\epsilon} ),
 $$
 in which the representation $\rho_{n-1}^{\epsilon}$ with $\epsilon = (-1)^n$
  is the  representation of $S_{n-1}$ having character 
 $(-1)^n (n-1)! \,\omega_{n-1}$,
viewing $S_{n-1}$ as the subset of    $S_n$ 
of all permutations holding the symbol $n$ fixed. 
\end{thm}

\begin{proof} 
(1) We treat  the cases of  $n$ even and $n$  odd separately.

Suppose first that $n=2m$. In this case (1) asserts
$$
 \rho_n^{(L)} :=  \mathrm{Ind}_{C_{n-1}}^{S_n} (\chi_{\triv}),
 $$
 with $C_n= < h> $ with $h \in S_n$ being a fixed $(n-1)$-cycle leaving letter $n$ fixed,
 say $h= (123 \cdots n-2\, n-1) (n).$
 It suffices to   compute the character $\psi_{2m}^{(L)}$ of $\rho_{2m}^{(L)} :=  \mathrm{Ind}_{C_{2m-1}}^{S_{2m}}( \chi_{\triv})$, 
 since the character determines a  representation.
Using the Frobenius formula \eqref{Frobenius} we have 
$$
\psi_{2m}^{(L)}(g) = \frac{1}{|C_{2m-1}|} \sum_{x \in S_{2m}} \hat{\chi} ( x^{-1} g x),
$$
where $\hat{\chi}$ is given by \eqref{517}.
The powers  $h^k$ have cycle structure
$[d^c, 1]$ with $cd= 2m-1$, so the character $\psi_{2m}^{(L)}$ is $0$ away from these conjugacy classes. The conjugacy class $C_{[d^c, 1]}$ is
of size $\frac{(2m)!}{d^c  c!}$ by \eqref{230b}. 
We deduce
$$
\psi_{2m}^{(L)}( C_{[ d^c, 1]} ) := \sum_{g \in C_{[ d^c, 1]}} \Big( \frac{1}{2m-1} \sum_{x \in S_{2m}} \hat{\chi} ( x^{-1} g x)\Big) 
$$
As in Theorem \ref{th51} we have 
$$
\psi_{2m}^{(L)} (C_{ [d^c, 1]}) = | \{ (g', x, h^j) :   h= x g' x^{-1} \, \mbox{with}  \quad g' \in C_{[d^c,1]} ,  h^j \in C_{n-1},  x\in S_n \, \}|.
$$
There are $|C_{[d^c, 1]}| = \frac{(2m)!}{d^c c!}$ choices for $g$, there are $\phi(d)$ choices for $h^j$, and for each
such pair there are $|N(\langle h^j \rangle)| =  d^c c!$ choices of $x$, 
where again   $|N(G)|$ denotes the cardinality of the
normalizer of the subgroup $G$ of $S_{2m}$.
Here $G$ is  a cyclic group of order $d$ generated by an element $h^j$ having cycle structure $[d^c, 1]$.
We obtain
$$
\psi_{2m}^{(L)} (C_{[d^c, 1]})= \frac{1}{2m-1} \Big( \frac{(2m)!}{d^c c!} \cdot  d^c c! \cdot \varphi(d) \Big) = (2m)! \frac{\varphi(d)}{2m-1}.
$$
On comparing this character with the formula for the class function $(2m)! \omega_{2m-1}(C_{[d^c]})$
implied by Theorem \ref{th32} and the fact that $\omega_{2m+1}(C_{\lambda^{'}}) =0$ for non-rectangular partitions, we find
$$
\psi_{2m}^{(L)}(C_{[\lambda^{'}, 1]})= (2m)! \omega_{2m-1}(C_{\lambda^{'}})  \quad \mbox{for all} \,\, \lambda^{'} \vdash 2m-1.
$$
In addition  $\psi_{2m}^{(L)}(C_{\lambda}) =0$ for all $\lambda \vdash 2m$ with no parts equal to $1$, so we have
$$
\psi_{2m}^{(L)}(C_{\lambda})= (2m)! \omega_{2m-1}^{\ast} ( C_{\lambda}) \quad \mbox{for all}\quad \lambda \vdash 2m,
$$
as asserted. Here $\epsilon=1$ and the  character $\psi_{2m}^{(L)}$ is nonnegative.

Suppose secondly  that $n=2m+1$ is odd,  in which   case (1) asserts
$$
 \rho_n^{(L)} :=  \mathrm{Ind}_{C_{2m}}^{S_{2m+1}} (\chi_{\sgn})
 $$
We proceed as above. All elements in $C_{[d^c, 1]}$ will be added with the same sign from the
character $\chi_{sgn} (C_{[d^c]}) = (-1)^c$, using the fact that $n-1=2m$ is even.  We find
\begin{eqnarray*}
\psi_{2m+1}^{(L)} (  C_{[d^c, 1]}) &= &(-1)^a \frac{1}{2m-1} \Big( \frac{(2m)!}{d^c c!} \cdot (2m+1) d^c c! \cdot \varphi(d) \Big)\\
& = &- (2m)!  \big( (-1)^{c+1}) \frac{\varphi(d)}{2m-1}.
\end{eqnarray*}
Comparing the right side with the formula of Theorem \ref{th32} yields
$$
\psi_{2m+1}^{(L)} (C_{[\lambda^{'},1]})  = - (2m)! \omega_{2m}( C_{\lambda^{'}}) \quad \mbox{for all} \quad \lambda^{'} \vdash n-1,
$$
In addition $\psi_{2m+1}^{(L)} (C_{\lambda})=0$ for all $\lambda \vdash n$ that have no part equal to $1$, so we have
$$
\psi_{2m+1}^{(L)}(C_{\lambda}) = - (2m+1)! \omega_{2m}^{\ast} ( C_{\lambda}) \quad \mbox{for all}\quad \lambda \vdash n,
$$
with $\epsilon = -1$, as asserted. 

The degree of $\rho_{n}^{(L)}$ is $\psi_{n}^{(L)}( C_{[1^n]}) = \frac{n!}{n-1} = n \cdot (n-2)!.$
It contains a copy of the one-dimensional representation $(\chi_{\sgn})^n$ on $S_n$, so is reducible for $n \ge 2$.

(2) We have from (1) by  transitivity of induction 
$$
 \rho_{n-1}^{+} :=  \mathrm{Ind}_{S_{n-1}}^{S_n} \big(  \mathrm{Ind}_{C_{n-1}}^{S_{n-1}}( (\chi_{\sgn})^n)\big).
$$
We  identify the inner sum representation with $\rho_{n-1}^{\epsilon}$, treating the
cases $n$ even and odd separately. 
For the case $n=2m$,  we use  Theorem \ref{th51} to find that  
the inner induced representation on the right is the 
 representation
$$
 \rho_{n-1}^{-} =  \mathrm{Ind}_{C_{n-1}}^{S_{n-1}}( \chi_{\triv}) =  \mathrm{Ind}_{C_{2m-1}}^{S_{2m-1}}( \chi_{\triv}),
 $$
 having character $(2m-1)! \omega_{2m-1}$.
 For the case $n=2m+1$ we use 
 Theorem \ref{th52} to find that the inner induced representation on the right is the
 representation 
 $$
 \rho_n^{+} =  \mathrm{Ind}_{C_{n-1}}^{S_{n-1}}(\chi_{\sgn}),
 $$
having character $- (2m)! \, \omega_{2m}$, which completes (2).
\end{proof} 


\subsection{The rescaled $1$-splitting measure $ n! \nu_{1, n}$ is the character of a virtual representation of $S_n$}\label{sec55}
 
 The results above imply that the $1$-splitting measure $\nu_{n,1}^{\ast}$ scaled by a factor $n!$ is
 the character of a virtual representation of $S_n$.
 

 \begin{thm}\label{th55}
 The class function $(-1)^{n-1} n! \nu_{n, 1}(g)$ for $g \in S_n$ of the rescaled $1$-splitting measure   is the character of a
 virtual  representation $\rho_{n, 1}$ of $S_n$.
 
  (1) For even $n=2m$, we have 
  $$
 \rho_{2m,1} = (\rho_{2m}^{-})^{-1} \oplus \rho_{2m-1}^{(L)}.
  $$
 
(2)  For odd $n=2m+1$ we have
$$
\rho_{2m+1,1} = \rho_{2m+1}^{+} \oplus (\rho_{2m}^{(L)})^{-1}.
$$
\end{thm}

\begin{proof}
On the character level we have
$$
n! \nu_{n, 1}= n! \omega_{n} + n! \omega_{n-1}^{\ast}.
$$
It follows from Theorems  \ref{th51},  \ref{th52} and \ref{th54}
that exactly one of the two terms on the right is the character of a 
representation and the other the negative of a character of a  representation.
The answers depend on the parity of $n$, as given.
\end{proof} 


\section{Splitting Measure for $z=-1$}\label{sec6}

We determine the splitting measure at $z=-1$, which has an especially simple structure.
 \begin{thm}\label{th61} 
 For $n \ge 2$ for $z=-1$ the  splitting measure  $\nu_{n, -1}^{\ast}$ is
 a nonnegative measure supported on the conjugacy classes $C_{\lambda}$
 with $\lambda = [1^n]$, the identity class, and $\lambda = [ 2, 1^{n-2}]$, the
 class of a $2$-cycle. It has equal mass $\nu_{n, -1}(C_{\lambda})= \frac{1}{2}$
 on these two classes.
 \end{thm}
 
 \begin{proof}
 By Lemma \ref{le22}{2} all necklace polynomials $M_m(-1)=0$ for $m \ge 3$.
 Also using this lemma, 
 $ { M_2(X) \choose m_2}$ for $m_2 \ge 2$ contains $M_2(X) -1$ in the numerator
 hence contributes a zero when $X=-1$. 
 Consequently  $\nu_{n, -1}^{\ast}(C_{\lambda}) =0$ if $\lambda$ is  a partition containing
 some part $3$ or larger, or containing at least two parts equal to $2$.
 The only allowable $\lambda$ are $[1^n]$ and $[2, 1^{n-2}]$. We compute directly
 $$
 \nu_{n, -1}^{\ast}(C_{ [1^n]}) = \frac{1}{(-1)^{n-1}(-2)}\cdot \frac{ (-1) (-2) \cdots (-n)}{n!} = \frac{(-1)^{2n}}{2} = \frac{1}{2}.
 $$
 Since the sum of all values is $1$ by Lemma \ref{le25} we must have
 $\nu_{n, -1}^{\ast} (C_{ [2, 1^{n-2}]} ) = \frac{1}{2}$ as well. 
 \end{proof}
 
 It is immediate that a scaled version of this measure is the character
 of a representation of $S_n$.
  \begin{thm}\label{th62} 
 For $n \ge 2$ the  rescaled $(-1)$-splitting measure  $n! \nu_{n, -1}^{\ast}$ is
 the character of a  representation ${\trho}_n$ of $S_n$.
 It  is realized as a permutation representation, given as the induced 
 representation $ \mathrm{Ind}_{C_2}^{S_n}(\chi_{\triv})$
 where $C_2 = \{ e, (12)\}$ is a group given
 by a $2$-cycle.
 \end{thm}

 \begin{proof}
 The calculation of the induced representation may be  done  similarly to that  in Theorem \ref{th51},
 and it agrees with that from Theorem \ref{th61}.
 \end{proof}

\section{Concluding Remarks }\label{sec7}

As noted in the introduction,  the elements of $S_n$ having conjugacy classes with
the partitions of the shapes in Theorem \ref{th41} (1)  are exactly the {regular elements} of
the Coxeter group $S_n$, in the sense of Springer \cite[Sec. 5.1]{Springer:1974}.
This fact  was established computationally  via the calculation in Lemma \ref{le24b}.
It would be of interest  to give a more conceptual ``geometric" explanation of the
appearance of the Springer regular elements in the limiting $1$-splitting distribution.

The Springer regular  elements appear  in the ``cyclic sieving phenomenon" of
Reiner, Stanton and White \cite{RSW:2004}; see Theorem 1.1 and Section 8 of
their paper. The cyclic sieving phenomenon provides a combinatorial interpretation of
certain polynomials $f(x) \in \ZZ[x]$ evaluated at $n$-th roots of unity, with these values 
being  integers. It has an interpretation in geometric representation theory found
by Fontaine and Kamnitzer \cite{FK:2014}. 
It might be interesting to  investigate the  values of $z$-splitting  measures on $S_n$
(rescaled by $n!$) for $z$ being an $n$-th  or $(n-1)$-st root of unity. 
In Section \ref{sec6} above we determined these measures for $z=-1$.
Sagan \cite{Sagan:2011} surveys further work on the cyclic sieving phenomenon. \medskip


\paragraph{\bf Acknowledgments.}
 I thank  Bruce Sagan for very helpful remarks  about Springer regular elements
  and the cyclic sieving phenomenon and for some corrections.  I thank
 David Speyer for  remarks  about induced representations
 and Trevor Hyde for many  helpful comments and corrections. 
Some results of this paper were presented
in November 2014  at the Hausdorff Institute at  Bonn in
the workshop ``Number Theory and Noncommutative Geometry".


\section{Appendix:Tables of Measures $\nu_{n,1}^{\ast}$ and $\omega_n, \omega_{n-1}^{\ast}$}\label{sec8}

The tables below give the  mass of the measures $\omega_n, \omega_{n-1}^{\ast}$ and $\nu_{n,1}^{\ast}$
on conjugacy classes of $S_n$ for $2 \le n \le 9$. Note that
$\omega_n( C_{\lambda}) = \sum_{g \in C_{\lambda}} \omega_n(g)$, 
so $\omega_n(g) = \frac{1}{|C_{\lambda}|} \omega_n(C_{\lambda})$, and similarly
for other measures.
The mass on each  element $g \in S_n$ falling in the 
category ``Other" is $0$. \\

%
%

\begin{minipage}{\linewidth}
\begin{center}
\begin{tabular}{|c || r | r |}
\hline
 $n=2$ & $C_{[1^2]}$ & $C_{[2]}$   \\
\hline
 $|C_{\lambda}|$     &   $1$     & $1$ \\
 \hline
 $\omega_2$         &  $-\frac{1}{2}$ & $\frac{1}{2}$\\
 \hline
 $\omega_1^{\ast}$ &      $1$               & $0$ \\
 \hline
 $\nu_{2, 1}^{\ast}$ & $\frac{1}{2}$ & $\frac{1}{2}$ \\
 \hline
\end{tabular} \par
\bigskip
\hskip 0.5in {\rm TABLE A-1: Symmetric group $S_2$.}  
\newline
\newline
\end{center}
\end{minipage}

%
%

\begin{minipage}{\linewidth}
\begin{center}
\begin{tabular}{|c || r | r ||r|}
\hline
 $n=3$ & $C_{[1^3]}$ & $C_{[3]}$  & $C_{[2,1]}$ \\
\hline
 $|C_{\lambda}|$     &   $1$     & $2$  & $3$\\
 \hline
 \hline
 $\omega_3$         &  $\frac{1}{3}$ & $\frac{2}{3}$ &  $0$\\
 \hline
 $\omega_2^{\ast}$ &      $-\frac{1}{2} $    & $0$            & $\frac{1}{2}$ \\
 \hline
 $\nu_{3, 1}^{\ast}$ & $-\frac{1}{6}$ & $\frac{2}{3}$& $\frac{1}{2}$\\
 \hline
\end{tabular} \par
\bigskip
\hskip 0.5in {\rm TABLE A-2: Symmetric group $S_3$.}  
\newline
\newline
\end{center}
\end{minipage}

%
%

\begin{minipage}{\linewidth}
\begin{center}
\begin{tabular}{|c || r | r |r ||r||r|}
\hline
 $n=4$ & $C_{[1^4]}$ & $C_{[2^2]}$  & $C_{[4]}$ & $C_{[3,1]}$\ & \mbox{Other}\\
\hline
 $|C_{\lambda}|$     &   $1$     & $3$  & $6$ & $8$ & $6$\\
 \hline
 \hline
 $\omega_4$         &  $-\frac{1}{4}$ & $-\frac{1}{4}$ & $\frac{1}{2}$ & $0$ & $0$\\
 \hline
 $\omega_3^{\ast}$ &      $\frac{1}{3} $    & $0$     & $0$       & $\frac{2}{3}$ & $0$\\
 \hline
 $\nu_{4, 1}^{\ast}$ & $\frac{1}{12}$ & $-\frac{1}{4}$  & $\frac{1}{2}$& $\frac{2}{3}$ & $0$ \\
 \hline
\end{tabular} \par
\bigskip
\hskip 0.5in {\rm TABLE A-3: Symmetric group $S_4$.}  
\newline
\newline
\end{center}
\end{minipage}

%
%

\begin{minipage}{\linewidth}
\begin{center}
\begin{tabular}{|c || r | r ||r |r||r|}
\hline
 $n=5$ & $C_{[1^5]}$ & $C_{[5]}$  & $C_{[2^2,1]}$ & $C_{[4,1]}$\ & \mbox{Other}\\
\hline
 $|C_{\lambda}|$     &   $1$     & $24$  & $15$ & $30$ & $50$\\
 \hline
 \hline
 $\omega_5$         &  $\frac{1}{5}$ & $\frac{4}{5}$ & $0$ & $0$ & $0$\\
 \hline
 $\omega_4^{\ast}$ &      $-\frac{1}{4} $    & $0$     & $-\frac{1}{4}$       & $\frac{1}{2}$ & $0$\\
 \hline
 $\nu_{5, 1}^{\ast}$ & $-\frac{1}{20}$ & $\frac{4}{5}$  & $-\frac{1}{4}$& $\frac{1}{2}$ & $0$ \\
 \hline
\end{tabular} \par
\bigskip
\hskip 0.5in {\rm TABLE A-4: Symmetric group $S_5$.}  
\newline
\newline
\end{center}
\end{minipage}

%
%

\begin{minipage}{\linewidth}
\begin{center}
\begin{tabular}{|c || r | r |r |r ||r || r|}
\hline
 $n=6$ & $C_{[1^6]}$ & $C_{[2^3]}$ & $C_{[3^2]}$ & $C_{[6]}$  & $C_{[5,1]}$  & \mbox{Other}\\
\hline
 $|C_{\lambda}|$     &   $1$     & $15$  & $40$ & $120$ & $144$ & 400\\
 \hline
 \hline
 $\omega_6$         &  $-\frac{1}{6}$ & $\frac{1}{6}$ & $-\frac{1}{3}$ & $\frac{1}{3}$ & $0$ & $0$\\
 \hline
 $\omega_5^{\ast}$ &      $\frac{1}{5} $    & $0$     & $0$ & $0$        & $\frac{4}{5}$ & $0$\\
 \hline
 $\nu_{6, 1}^{\ast}$ & $\frac{1}{30}$ & $\frac{1}{6}$  & $-\frac{1}{3}$& $\frac{1}{3}$ & $\frac{4}{5}$ & $0$ \\
 \hline
\end{tabular} \par
\bigskip
\hskip 0.5in {\rm TABLE A-5: Symmetric group $S_6$.}  
\newline
\newline
\end{center}
\end{minipage}

%
%

\begin{minipage}{\linewidth}
\begin{center}
\begin{tabular}{|c || r | r ||r |r |r || r|}
\hline
 $n=7$ &   $C_{[1^7]}$ & $C_{[7]}$ & $C_{[2^3, 1]}$ & $C_{[3^2, 1]}$ & $C_{[6,1]}$    & \mbox{Other}\\
\hline
 $|C_{\lambda}|$     &   $1$     & $120$  & $105$ & $280$ & $840$ & $3694$\\
 \hline
 \hline
 $\omega_7$         &  $\frac{1}{7}$ & $\frac{6}{7}$ & $0$ & $0$ & $0$ & $0$\\
 \hline
 $\omega_6^{\ast}$ &      $-\frac{1}{6} $    & $0$     & $\frac{1}{6}$ & $-\frac{1}{3}$        & $\frac{1}{3}$ & $0$\\
 \hline
 $\nu_{7, 1}^{\ast}$ & $-\frac{1}{42}$ & $\frac{6}{7}$  & $\frac{1}{6}$& $-\frac{1}{3}$ & $\frac{1}{3}$ & $0$ \\
 \hline
\end{tabular} \par
\bigskip
\hskip 0.5in {\rm TABLE A-6: Symmetric group $S_7$.}  
\newline
\newline
\end{center}
\end{minipage}

%
%

\begin{minipage}{\linewidth}
\begin{center}
\begin{tabular}{|c || r | r |r |r ||r || r|}
\hline
 $n=8$ &   $C_{[1^8]}$ & $C_{[2^4]}$ & $C_{[4^2]}$ & $C_{[8]}$ & $C_{[7,1]}$    & \mbox{Other}\\
\hline
 $|C_{\lambda}|$     &   $1$     & $105$  & $1260$ & $5040$ & $5760$ & $28154$\\
 \hline
 \hline
 $\omega_8$         &  $-\frac{1}{8}$ & $-\frac{1}{8}$ & $-\frac{1}{4}$ & $\frac{1}{2}$ & $0$ & $0$\\
 \hline
 $\omega_7^{\ast}$ &      $\frac{1}{7} $    & $0$     & $0$ & $0$        & $\frac{6}{7}$ & $0$\\
 \hline
 $\nu_{8, 1}^{\ast}$ & $\frac{1}{56}$ & $-\frac{1}{8}$  & $-\frac{1}{4}$& $\frac{1}{2}$ & $\frac{6}{7}$ & $0$ \\
 \hline
\end{tabular} \par
\bigskip
\hskip 0.5in {\rm TABLE A-7: Symmetric group $S_8$.}  
\newline
\newline
\end{center}
\end{minipage}

%
%

\begin{minipage}{\linewidth}
\begin{center}
\begin{tabular}{|c || r | r |r |r |r |r ||r|}
\hline
 $n=9$ &   $C_{[1^9]}$ & $C_{[3^3]}$ &  $C_{[9]}$    & $C_{[2^4,1]}$ & $C_{[4^2,1]}$ & $C_{[8, 1]}$ &  \mbox{Other}\\
\hline
 $|C_{\lambda}|$     &   $1$     & $2240$  & $40320$ & $945$ & $11340$ & $45360$ & $262674$\\
 \hline
 \hline
 $\omega_9$         &  $\frac{1}{9}$ & $\frac{2}{9}$ & $\frac{2}{3}$ & $0$ & $0$ & $0$ & $0$\\
 \hline
 $\omega_8^{\ast}$ &      $-\frac{1}{8} $    & $0$     & $0$ & $-\frac{1}{8}$        & $-\frac{1}{4}$ & $\frac{1}{2}$ & $0$\\
 \hline
 $\nu_{9, 1}^{\ast}$ & $-\frac{1}{72}$ & $\frac{2}{9}$  & $\frac{2}{3}$ & $-\frac{1}{8}$ &$-\frac{1}{4}$& $\frac{1}{2}$ & $0$ \\
 \hline
\end{tabular} \par
\bigskip
\hskip 0.5in {\rm TABLE A-8: Symmetric group $S_9$.}  
\newline
\newline
\end{center}
\end{minipage}


\begin{thebibliography}{99}
\newcommand{\au}[1]{{#1},}
\newcommand{\ti}[1]{\textit{#1},}
\newcommand{\jo}[1]{{#1}}
\newcommand{\vo}[1]{\textbf{#1}}
\newcommand{\yr}[1]{(#1),}
\newcommand{\pp}[1]{#1.}
\newcommand{\pps}[1]{#1;}
\newcommand{\bk}[1]{{#1},}
\newcommand{\inbk}[1]{in: {#1}}
\newcommand{\xxx}[1]{{arXiv:#1}}


\bibitem{Bhargava:2007}
M.~Bhargava.
\newblock Mass formulae for extensions of local fields, and conjectures on the
  density of number field discriminants.
\newblock {\em Int. Math. Res. Not. IMRN}, {\bf 2007}, no. 17, Art. ID rnm052, 20pp.


\bibitem{CEF:2013}
T. Church, J. S. Ellenberg and B. Farb,
\newblock Representation stability in cohomology and asymptotics for
families of varieties over finite fields,
\newblock Algebraic topology: applications and new directions, 1--54.
Contemp. Math. 620, Amer. Math. Soc., Providence, RI, 2014.

\bibitem{CEF:2013b}
T. Church, J. S. Ellenberg and B. Farb,
\newblock FI-modules and stability for representations of symmetric groups.
Duke Math. J. {\bf 164} (2015), no. 9, 1833--1910.



\bibitem{Cohen:1970}
S. ~D. Cohen,
\newblock The distribution of polynomials over finite fields.
\newblock {\em Acta Arithmetica}, 17, 255--271, 1970.



\bibitem{Cohn:2004}
Henry Cohn, 
\newblock{ Projective geometry over $\FF_1$ and the Gaussian binomial
coefficients,}
Amer. Math. Monthly {\bf 111} (2004), No. 6. 487--495.


\bibitem{Connes-C:2010}
A. Connes, and C. Consani, 
\newblock{Schemes over  $\FF_1$ and zeta functions,}
\newblock{ Compositio Math. {\bf  146} (2010), no. 6, 1383--1415.}



\bibitem{Connes-C:2011}
A. Connes, and C. Consani, 
\newblock{On the notion of  geometry  over  $\FF_1$,}
\newblock{J. Algebraic Geometry {\bf 20 } (2011), no. 3, 529--557.}


\bibitem{Connes-C-M:2009}
A. Connes, C. Consani, and M. Marcolli,
\newblock{Fun with $\FF_1$,}
\newblock{ J. Number Theory {\bf  129} (2009), no. 6, 1532--1561.}







\bibitem{Deitmar:2005}
A. Deitmar,
\newblock{Schemes over $\FF_1$,}
\newblock{ in; {\em Number Fields and Function Fields: Two Parallel Worlds,}}
\newblock{Prog. Math. No. {\bf 239}, Birkh\"{a}user Boston: Boston, MA 2005.}


\bibitem{FK:2014}
B. Fontaine and J. Kamnitzer,
{\em Cyclic sieving, rotation, and geometric representation theory,}
Selecta Math. (N.S.) {\bf 20} (2014), no. 2, 609--625.

\bibitem{Frei:2007}
G. Frei,
\newblock{ The Unpublished Section Eight: On the Way to Function Fields
over a Finite Field,}
\newblock Sect. II.4 in: {\em The Shaping of Arithmetic after  C. F. Gauss' Disquisitiones Arithmeticae,}
(C. Goldstein, N. Shappacher, J. Schwermer, Eds.), pp. 159--198,
\newblock: {Springer: New York 2007}




\bibitem{Fulman:2002}
J. Fulman,
\newblock Random matrix theory over finite fields,
\newblock {\em Bull. Amer. Math. Soc. (N.S.)} {\bf 39} (2002), no. 1, 51--85.


\bibitem{FH:1991}
W. Fulton and J. Harris,
\newblock{Representation Theory. A First Course,}
\newblock{Springer-Verlag: New York 1991}

\bibitem{GHU:2011}
H.-C. Graf von Bothmer, L. Hinsch and U. Stuhler,
\newblock{Vector bundles over projective spaces. The case $\FF_1$.}
\newblock{Arch. Math. (Basel) {\bf 96} (2011), no. 3, 227--234.}


\bibitem{Lagarias-W:2014}
J.~C. Lagarias and B. ~L. Weiss,
\newblock{Splitting behavior of $S_n$ polynomials,}
Research in Number Theory {\bf 1} (2015), no. 7, 30pp.
DOI 10.1007/s40993-015-0006-6.










\bibitem{Lopez-L:2011}
J. ~L\'{o}pez  Pe{n}a and O. Lorscheid,
\newblock{ Mapping $\FF_1$-land: an overview of geometries over the field with one element,}
\newblock{in: {\em Noncommutative geometry, arithmetic, and related topics,} 241--265.}
\newblock{Johns Hopkins University Press: Baltimore, MD, 2011}

\bibitem{Lorscheid:2011}
O. Lorscheid,
\newblock{Torified varieties and their geometry over $\FF_1$,}
\newblock{Math. Z.{\bf  267} (2011), no. 3-4, 605--643.}

\bibitem{Lorscheid:2012}
O. Lorscheid,
\newblock{Algebraic groups over the field with one element,}
\newblock{Math. Z.. {\bf 271} (2012), no. 1-2, 117--138.}



\bibitem{Macdonald:1995}
I. G. Macdonald,
\newblock {\em Symmetric Functions and Hall Polynomials. Second Edition,}
\newblock {Oxford University Press: Oxford 1995.}





\bibitem{Metropolis:1983}
N. Metropolis and G.-C. Rota.
\newblock{Witt vectors and the algebra of necklaces,}
\newblock{Advances in Math.} 50,  95--125, 1983.


\bibitem{Moreau:1872}
C. ~Moreau.
\newblock{Sur les permutations circulaires distinctes,}
Nouvelles annales de math\'{e}matiques, journal des
candidats aux \'{e}coles polytechnique et normale, S\'{e}r. 2, {\bf 11} (1872), 309--314.

\bibitem{RPM:2011}
S.Ramdorai, J. Plazas and M. Marcolli,
\newblock{Introduction to Motives. With an appendix by M. Marcolli,}
pp. 41-88 in: {\em Noncommutative Geometry and Physics: Renormalisation,
Motives, Index Theory,} (A. Carey, Ed.), European Math. Soc. Press 2011. 

\bibitem{RSW:2006}
V. Reiner, D. Stanton and P. Webb,
\newblock{Springer's regular elements over arbitrary fields,}
\newblock{Math. Proc. Cambridge Phil. Soc. {\bf 141} (2006), no. 2, 209--229.}

\bibitem{RSW:2004}
V. Reiner, D. Stanton and D. White,
\newblock{The cyclic sieving phenomenon,}
\newblock J. Combin. Theory Ser. A {\bf 108} (2004), no. 1, , 17--50.

\bibitem{Rosen:2002}
M.~Rosen.
\newblock {\em Number theory in function fields}, volume 210 of {\em Graduate
  Texts in Mathematics}.
\newblock Springer-Verlag, New York, 2002.


\bibitem{Sagan:2001}
B. Sagan,
\newblock{The symmetric group. Representations, combinatorial algorithms and
symmetric functions. Second edition.} Graduate Texts in Mathematics, 203.
Springer-Verlag: New York 2001. xvi+238pp.
\bibitem{Sagan:2011}
B. Sagan,
\newblock{The cyclic sieving phenomenon: a survey,}
\newblock{pp. 183--234 in: {\em Surveys in Combinatorics 2011,}
\newblock{ (R. Chapman, Ed.), London Math. Soc. Lecture Notes, Vol. 392,}
\newblock Cambridge University Press: Cambridge 2011.}


\bibitem{Soule:2004}
C. Soul\'{e},
\newblock{ Les vari\'{e}ties sur las corps \'{a} un \'{e}l\'{e}ment,}
\newblock{ Moscow Math. J.  {\bf 4} (2004), no. 1, 217--244, 312.}




\bibitem{Springer:1974}
T. Springer,
\newblock{Regular elements of finite reflection groups,}
\newblock{Inventiones Math. {\bf 25} (1974), 159--198.}





\bibitem{Stanley:1997}
R.~P. Stanley.
\newblock {\em Enumerative combinatorics. {V}ol. 1}, volume~49 of {\em
  Cambridge Studies in Advanced Mathematics}.
\newblock Cambridge University Press, Cambridge, 1997.
\newblock Corrected reprint of the 1986 original.

\bibitem{Stanley:1999}
R.~P. Stanley.
\newblock {\em Enumerative combinatorics. {V}ol. 2}, volume~62 of {\em
  Cambridge Studies in Advanced Mathematics}.
\newblock Cambridge University Press, Cambridge, 1999.

\bibitem{Szczesny:2012}
M. Szczesny,
\newblock{ Representations of quivers over $\FF_1$ and Hall algebras,}
\newblock{Int. Math. Res. Not. IMRN {\bf 2012}, no. 10, 2377-2404.}
\bibitem{Tits:1957}
J. Tits, 
\newblock{Sur les analogues alg\'{e}briques des groupes semi-simples complexes,}
\newblock{ {\em Colloque d'alg\'{e}bre sup\`{e}rieure, tenu \'{a} Bruxelles du 19 au 22 d\'{e}cembre 1956,}}
\newblock{ Centre Recherces Math. \'{E}tablissements Ceuterick, Louvain, Paris: Librairie Gauthier-Villars 1957,
pp. 261--289.}

\bibitem{Toen:2009}
B. To\"{e}n,
\newblock{Au-dessous de ${\rm Spec}(\ZZ)$,}
\newblock{J. of K-Theory {\bf 3} (2009), no. 3, 437--500.}

\bibitem{Weiss:2013}
B.~L.~Weiss,
\newblock{Probabilistic Galois theory over $p$-adic fields,}
\newblock{{\em J. Number Theory}  {\bf 133} (2013), 1537--1563.}

\end{thebibliography}
\end{document}